\documentclass{aims}

\usepackage{dsfont}
\usepackage{url}
\usepackage[utf8]{inputenc}
\usepackage[T1]{fontenc}
\usepackage{lmodern}
\usepackage{babel}
\usepackage{mathtools}
\usepackage{amssymb}
\usepackage{lipsum}
\usepackage{mathrsfs}
\usepackage{color}

\usepackage{amsmath}
\usepackage{paralist}
\usepackage[misc]{ifsym}
\usepackage{epsfig}
\usepackage{epstopdf}
\usepackage[colorlinks=true]{hyperref}

\hypersetup{urlcolor=blue, citecolor=red}

\allowdisplaybreaks

\def\currentvolume{X}
 \def\currentissue{X}
  \def\currentyear{200X}
   \def\currentmonth{XX}
    \def\ppages{X-XX}
     \def\DOI{10.3934/xx.xxxxxxx}

\newtheorem{theorem}{Theorem}[section]
\newtheorem{corollary}[theorem]{Corollary}

\newtheorem{lemma}[theorem]{Lemma}
\newtheorem{proposition}[theorem]{Proposition}

\theoremstyle{definition}

\newtheorem{remark}[theorem]{Remark}

\title[Stability of determining the potential from partial boundary]
{Stability of determining the potential from partial boundary data in a Schr\"odinger equation in the high frequency limit}
\author[Mourad Choulli]{}

\subjclass{Primary: 35R30.}

\keywords{Schr\"odinger equation in the high frequency limit, potential, partial Dirichlet-to-Neumann map, stability inequality, quantitative unique continuation, quantitative Runge approximation.}

\begin{document}
\maketitle

\centerline{\scshape
Mourad Choulli$^{{\href{mailto:mourad.choulli@univ-lorraine.fr}{\textrm{\Letter}}}}$
}

\medskip

{\footnotesize
\centerline{Universit\'e de Lorraine, France}
}

\bigskip


\begin{abstract}
We establish stability inequalities for the problem of determining the potential, appearing in a Sch\"odinger equation, from  partial boundary data in the high frequency limit. These stability inequalities hold under the assumption that the potential is known near the boundary of the domain under consideration.
\end{abstract}

\section{Introduction}\label{section1}

\subsection{Assumptions and notations}

Let $\mathrm{M}$ be a connected compact  $C^{1,1}$\linebreak Riemannian manifold, of dimension $n\ge3$, with boundary $\partial \mathrm{M}$ and endowed with $C^{2,1}$ metric tensor $g$.

Throughout this text, $\Delta=\Delta_g$ denotes the usual Laplace-Beltrami operator which is given in a local coordinates system by
\[
\Delta u=\frac{1}{\sqrt{|g|}}\sum_{j,k=1}^n\frac{\partial}{\partial x_j}\left(\sqrt{|g|}g^{jk}\frac{\partial}{\partial x_k}u \right).
\]

Recall that the space $H_\Delta(\mathrm{M})$ is defined as follows
\[
H_\Delta (\mathrm{M})=\{u\in H^1(\mathrm{M});\; \Delta u\in L^2(\mathrm{M})\},
\]
 and for $q\in L^\infty(\mathrm{M},\mathbb{R})$, consider the operator
\[
A_qu=(-\Delta +q)u,\quad u\in D(A)=H_\Delta (\mathrm{M})\cap H_0^1 (\mathrm{M}).
\]
The bounded sesquilinear form associated to $A_q$ is given by
\[
\mathfrak{a}_q(u,v)=\int_{\mathrm M} (\nabla u |\overline{\nabla v})_gd\mu +\int_{\mathrm{M}}qu\overline{v}d\mu,\quad u,v\in H_0^1(\mathrm{M}),
\]
where $d\mu$ is the Riemannian measure on $\mathrm{M}$ and $(\cdot|\cdot)_g$ is the scalar product associated to the metric $g$. That is we have
\[
(A_qu|v)=\mathfrak{a}_q(u,v),\quad u\in D(A),\; v\in H_0^1(\mathrm{M}).
\]
Here and henceforth, $(\cdot|\cdot)$ denotes the scalar product of $L^2(\mathrm{M})$.

It is a classical result that $\sigma(A_q)$, the spectrum of $A_q$, consists in a sequence of eigenvalues
\[
-\infty < \lambda_q^1\le \ldots \lambda_q^k\le \ldots \quad \mbox{and}\quad \lambda_q^k\rightarrow \infty \; \mbox{as}\; k\rightarrow \infty.
\]
Moreover, there exists $(\phi_q^k)$ an orthonormal basis of $L^2(\mathrm{M})$ consisting of \linebreak eigenfunctions, where each $\phi_q^k$ is an eigenfunction for $\lambda_q^k$ and, according to the elliptic regularity, $\phi_q^k\in H^2(\mathrm{M})$.

The resolvent of $A_q$ is defined on the resolvent set $\rho(A_q)=\mathbb{C}\setminus \sigma(A_q)$ according to the formula
\[
R_q(\mu)=(A_q-\mu)^{-1},\quad \mu \in \rho(A_q).
\]

Fix $q_0\in L^\infty(\mathrm{M},\mathbb{R})$, $\kappa_0>\|q_0\|_{\infty}$ and $\lambda\in   \rho(A_{q_0})$ satisfying $\lambda \ge \lambda_0$, where $\lambda_0>0$ is some constant. Then consider the set
\[
\mathscr{Q}_\lambda=\left\{q=q_0+q'\in L^\infty(\mathrm{M},\mathbb{R});\; \|q'\|_{L^\infty(\mathrm{M})}< \min \left(\|R_{q_0}(\lambda)\|^{-1},\kappa_1\right)\right\},
\]
where $\|R_{q_0}(\lambda)\|$ is the operator norm of $R_{q_0}(\lambda)$ in $\mathscr{B}(L^2(\mathrm{M}))$ and $\kappa_1>0$ is a an arbitrary fixed constant.

A usual perturbation argument shows that
\[
\lambda\in \bigcap_{q\in \mathscr{Q}_\lambda} \rho(A_q).
\]
Let
\[
\sigma_\lambda := \bigcup_{q\in \mathscr{Q}_\lambda}\sigma(A_q),
\]
define
\[
\mathbf{e}_\lambda=\max\left(\frac{1}{\mathrm{dist}(\lambda ,\sigma_\lambda)},1\right),
\]
and set $\kappa = \kappa_0+\kappa_1$.

In the rest of this text, $\Gamma $ and $\Sigma$ will denote two nonempty open subsets of $\partial \mathrm{M}$, and  $\mathrm{M}_0, \mathrm{M}'_0\subset \mathrm{Int}(\mathrm{\mathrm{M}})$ will denote two Riemannian submanifolds with $C^{1,1}$ boundary. Further,  assume that $\mathrm{M}_1=\mathrm{M}\setminus\mathrm{Int}(\mathrm{M}_0)$ is connected and $\mathrm{M}_0\subset \mathrm{Int}(\mathrm{M}'_0)$.

\subsection{Dirichlet-to-Neumann maps}

For every $s>0$, consider the closed subspace of $H^s(\partial \mathrm{M})$ given by
\[
H_\Gamma^s(\partial \mathrm{M})=\{ \varphi\in H^s(\partial \mathrm{M});\; \mathrm{supp}(\varphi)\subset \overline{\Gamma}\}.
\]
In the sequel, we equip $H_\Gamma^s(\partial \mathrm{M})$ with the norm of $H^s(\partial \mathrm{M})$.

For $q\in \mathscr{Q}_\lambda$ and $\varphi\in H^{3/2}(M)$, we denote hereafter by $u_{q,\lambda}(\varphi)\in H^2(\mathrm{M})$ the solution of the BVP
\[
(\Delta +\lambda -q)u=0\; \mathrm{in}\; \mathrm{M},\quad u_{|\partial \mathrm{M}}=\varphi.
\]
Then define  the partial Dirichlet-to-Neumann maps $\Lambda_{q,\lambda}^0$ and $\Lambda_{q,\lambda}^1$ respectively as follows
\begin{align*}
&\Lambda_{q,\lambda}^0: H_\Gamma^{3/2}(\partial \mathrm{M})\rightarrow L^2(\Sigma):\varphi\mapsto \partial_\nu u_{q,\lambda}(\varphi)_{|\Sigma}.
\\
&\Lambda_{q,\lambda}^1: H^{3/2}(\partial \mathrm{M})\rightarrow L^2(\Sigma):\varphi\mapsto \partial_\nu u_{q,\lambda}(\varphi)_{|\Sigma}.
\end{align*}
According to the usual a priori estimate in $H^2(\mathrm{M})$, we have
\[
\Lambda_{q,\lambda}^0\in \mathscr{B}(H_\Gamma^{3/2}(\partial \mathrm{M}),L^2(\Sigma))\quad \mathrm{and}\quad \Lambda_{q,\lambda}^1\in \mathscr{B}(H^{3/2}(\partial \mathrm{M}),L^2(\Sigma)).
\]

\subsection{Main results}

In this subsection, $\mathrm{M}$ is a $C^{1,1}$ bounded domain of $\mathbb{R}^n$ equipped with the round metric.

The first main result in this paper is the following theorem, where
\begin{align*}
&\mathbf{b}_\mu=\sqrt{2\cosh \left(\sqrt{\mu}/2\right)},\quad \mu >0,
\\
&\mathbf{m}_\lambda=\max\left( \lambda^{7/2}\mathbf{e}_\lambda^{3/2},\mathbf{b}_\lambda\right)\lambda^7\mathbf{e}_\lambda^2.
\end{align*}

\begin{theorem}\label{mt}
Let $n=3$. There exist two constants $C>0$ and $0<\Upsilon_0<e^{-1}$, only depending on $\mathrm{M}$, $\mathrm{M}_0$, $\kappa$, $\lambda_0$, $\Gamma$ and  $\Sigma$, such that, for any $q_j\in \mathscr{Q}_\lambda$, $j=1,2$ satisfying $q_1=q_2$ in $\mathrm{M}_1$ and $0<\| \Lambda_{q_1,\lambda}^0-\Lambda_{q_2,\lambda}^0\|< \Upsilon_0$, we have
\[
\|q_1-q_2\|_{H^{-1}(\mathrm{M})}\le C \mathbf{m}_\lambda\left|\ln \ln \| \Lambda_{q_1,\lambda}^0-\Lambda_{q_2,\lambda}^0\|\right|^{-2/(n+2)},
\]
where $\| \Lambda_{q_1,\lambda}^0-\Lambda_{q_2,\lambda}^0\|$ denotes the operator norm of  $\Lambda_{q_1,\lambda}^0-\Lambda_{q_2,\lambda}^0$.
\end{theorem}

\begin{remark}\label{remark1}
{\rm
It is worth observing that $\mathbf{b}_\mu=O(e^{\sqrt{\mu}/4})$ as $\mu$ goes to $\infty$. Also, from the proof of Theorem \ref{theorem1.0} one can see that $C\mathbf{m}_\lambda$ in Theorem \ref{mt} may be substituted by $C_\delta \mathbf{m}^\delta_\lambda$, where $0<\delta <1/2$ is taken arbitrary, $C_\delta$ is a constant depending on $\delta$ and $\mathbf{m}^\delta_\lambda$ is defined similarly to $\mathbf{m}_\lambda$ in which $\mathbf{b}_\lambda$ is replaced by $\mathbf{b}_\lambda^\delta =\sqrt{\sinh\sqrt{\lambda}/\sinh(\sqrt{\lambda}(1-2\delta))}$. Clearly, $\mathbf{b}^\delta_\lambda$ converges to $1$ as $\delta$ goes to $0$, and one can expect that $C_\delta$ converges to $\infty$ when $\delta$ tends to $0$.
}
\end{remark}

\begin{remark}\label{remark2}
{\rm
The restriction $n=3$ in Theorem \ref{mt} comes from the fact that we use a quantitative uniqueness of continuation result needing H\"older continuity of $H^2$-solutions. Unfortunately, according to Sobolev's embedding theorems, this property holds only if $n=3$. Precisely, $H^2(\mathrm{M})$ is continuously embedded in $C^{0,1/2}(\overline{\mathrm{M}})$ in dimension three.
}
\end{remark}

Without the assumption $q_1=q_2$ in $\mathrm{M}_1$, the stability inequality in Theorem \ref{mt} still an open problem. We point out that in \cite[Theorem 1]{CKS} and \cite[Theorem 1.1]{CDR}  double logarithmic stability inequalities were established for particular $\Gamma$ and $\Sigma$, but without assuming that $q_1=q_2$ in $\mathrm{M}_1$ and $n=3$. These results do not mention the dependence on $\lambda$.

Observe that Theorem \ref{mt} gives a precise dependence of the stability constant in term of the frequency $\omega=\sqrt{\lambda}$. It is worth noticing that even if $\lambda$ is small $\mathbf{e}_\lambda$ may be arbitrarily large when $\lambda$ is sufficiently close to $\sigma_\lambda$. In other words, the stability deteriorate not only for a large frequency but also if the square of the frequency is close to $\sigma_\lambda$.

\begin{remark}\label{remark3}
{\rm
It seems very difficult to substitute the distance to $\sigma_\lambda$ by a more explicit assumption. This problem is in fact related to the problem of establishing a lower bound of the gap between consecutive eigenvalues. To our knowledge there is no result concerning this issue even for $A_0$ (that is, $A_0=A_q$ with $q=0$) and under the assumption the eigenvalues of $A_0$ are simple. All that can be found in the literature is an upper bound of the gap between consecutive eigenvalues. We refer for instance to \cite[Theorem 1.1]{CZY}, where the following estimate was established
\[
\lambda_0^{k+1}-\lambda_0^k\le ck^{1/n},\quad k\ge 1.
\]
Here, the constant $c>0$ only depends on $n$ and $\mathrm{M}$.

Let us explain with a simple example the difficulty of obtaining  a lower bound of the gap between consecutive eigenvalues.
To this end, we suppose that $\mathrm{M} = \prod_{j=1}^n(0,\mu _j\pi )$, where the sequence $(\mu _1,\ldots \mu _n)$ is non-resonant (that is, every \linebreak nontrivial rational linear combination of $\mu _1,\ldots \mu _n$ is different from zero). From \cite[Proposition 5]{PS}, $A_0$ has simple eigenvalues given by
\[
\beta_K=\prod_{j=1}^n\frac{k_j^2}{\mu _j^2},\quad  K=(k_1,\ldots ,k_n)\in \mathbb{N}^n.
\]
Here $\mathbb{N}$ is the set of positive integers.

Although the eigenvalues of $A_0$ are giving explicitly, it is not an easy task to know how to deduce $(\lambda_0^k)$ from $(\beta_K)$. Therefore, even in this simple case we can not derive easily a lower bound of the gap between consecutive eigenvalues.
}
\end{remark}

We make a final remark.

\begin{remark}\label{remark4}
{\rm
 We observe that the mapping
\[
q\in \mathscr{Q}_\lambda \mapsto \Lambda_{q,\lambda}^0\in \mathscr{B}(H_\Gamma^{3/2}(\partial \mathrm{M}),L^2(\Sigma))
\]
 is Lipschitz continuous. Precisely, proceeding similarly to the proof of Lemma \ref{lemma1} we demonstrate the following inequality
\[
\| \Lambda_{q_1,\lambda}^0-\Lambda_{q_2,\lambda}^0\| \le C\lambda^2\mathbf{e}_\lambda^2\|q_1-q_2\|_{L^\infty(\mathrm{M})},\quad q_1,q_2\in \mathscr{Q}_\lambda,
\]
where the constant $C>0$ only depends on $n$, $\mathrm{M}$, $\kappa$ and $\lambda_0$.

This inequality shows that $\| \Lambda_{q_1,\lambda}^0-\Lambda_{q_2,\lambda}^0\|$ may be large for large $\lambda$ even if $\|q_1-q_2\|_{L^\infty(\mathrm{M})}$ is small.
}
\end{remark}

Some ideas we used to prove Theorem \ref{mt} were borrowed from \cite{GRZ,KU,RS}. We remark that the stability results in these references are stated in a different form from that appearing in Theorem \ref{mt}. Precisely, the modulus on continuity in Theorem \ref{mt} is given by $C\mathbf{m}_\lambda |\ln\ln t|^{-2(n+2)}$, $0<t<e^{-1}$, where $C$ is a constant depending on $\mathrm{M}$, $\mathrm{M}_0$, $\kappa$, $\lambda_0$, $\Gamma$ and  $\Sigma$. While the modulus of continuity in \cite{GRZ,KU,RS} is a function depending on $\mathrm{M}$, $\mathrm{M}_0$, $\kappa$, $\lambda_0$, $\Gamma$ and  $\Sigma$, and also on  $\lambda$.

Without any restriction on the dimension, we have the following result, where
\[
\tilde{\mathbf{m}}_\lambda=\lambda^2 \mathbf{b}_\lambda \mathbf{e}_\lambda.
\]

\begin{theorem}\label{mt1}
There exist two constants $C>0$ and $0<\Upsilon_1<1$, only depending on $n$, $\mathrm{M}$, $\mathrm{M}_0$, $\kappa$,  $\lambda_0$ and  $\Sigma$, such that, for any $q_j\in \mathscr{Q}_\lambda$, $j=1,2$, satisfying $q_1=q_2$ in $\mathrm{M}_1$ and $0<\| \Lambda_{q_1,\lambda}^1-\Lambda_{q_2,\lambda}^1\|< \Upsilon_1$, we have
\[
\|q_1-q_2\|_{H^{-1}(\mathrm{M})}\le C\tilde{\mathbf{m}}_\lambda \left|\ln \| \Lambda_{q_1,\lambda}^1-\Lambda_{q_2,\lambda}^1\|\right|^{-2/(n+2)}.
\]
\end{theorem}

The rest of this text is organized as follows. In Section \ref{section2} we give a quantification of the unique continuation from Cauchy data adapted to our context. A quantitative Range approximation result is proved in Section \ref{section3}. This result is essential in the proof of Theorem \ref{mt}. Section \ref{section4} is devoted to establishing a preliminary inequality which can be considered as an intermediate result that we need to prove Theorem \ref{mt}. In Section \ref{section5} we show that the preliminary inequality of Section \ref{section4} can be used to obtain a uniqueness result for the inverse problem of determining the potential or the conformal factor from partial boundary data. Theorems \ref{mt} and \ref{mt1} are proved in Section \ref{section6} with the aid of complex geometric optic solutions. Finally, in Section \ref{section7} we modify the analysis we carried out in the case of the Dirichlet boundary condition to extend Theorems \ref{mt} and \ref{mt1} to the case of an impedance boundary condition.

\section{Quantitative unique continuation from Cauchy data}\label{section2}

The following lemma will be useful in the sequel.

\begin{lemma}\label{lemma1}
Pick $q\in \mathscr{Q}_\lambda$.

\noindent
$(i)$ For $f\in L^2(\mathrm{M})$, we have
\begin{equation}\label{1}
\|R_q(\lambda)f\|_{H^j(\mathrm{M})}\le c_1\lambda^{j/2}\mathbf{e}_\lambda\|f\|_{L^2(\mathrm{M})},\quad j=0,1,2.
\end{equation}

\noindent
$(ii)$ For $\varphi\in H^{3/2}(\partial \mathrm{M})$, we have
\begin{equation}\label{1.0}
\|u_{q,\lambda}(\varphi)\|_{H^2(\mathrm{M})}\le c_2\lambda^2\mathbf{e}_\lambda\|\varphi\|_{H^{3/2}(\partial \mathrm{M})}.
\end{equation}

\noindent
Here, the constants $c_1>0$ and $c_2>0$ only depend on $n$, $\mathrm{M}$, $\kappa$ and $\lambda_0$.
\end{lemma}

\begin{proof}
(i) Let $f\in L^2(\mathrm{M})$  and $u=R_q(\lambda)f$. Then the usual a priori estimate in $H^2(\mathrm{M})$ yields
\begin{equation}\label{2}
\|u\|_{H^2(\mathrm{M})}\le c_0(\kappa +\lambda)\|u\|_{L^2(\mathrm{M})}\le c_0(\kappa/\lambda_0 +1)\lambda \|u\|_{L^2(\mathrm{M})},
\end{equation}
where $c_0>0$ is a constant only depending on $n$ and $\mathrm{M}$.

On the other hand, we have
\[
u=\sum_{k\ge 1}\frac{(f|\phi_q^k)}{\lambda_q^k-\lambda}\phi_q^k,
\]
and hence
\[
\|u\|_{L^2(\mathrm{M})}^2=\sum_{k\ge 1} \frac{|(f|\phi_q^k)|^2}{|\lambda_q^k-\lambda|^2}.
\]
In consequence, the following inequality holds
\begin{equation}\label{3}
\|u\|_{L^2(M)}\le \mathbf{e}_\lambda\|f\|_{L^2(\mathrm{M})}.
\end{equation}
Then \eqref{3} in \eqref{2} gives \eqref{1} for $j=2$. The case $j=1$ can be easily deduced from \eqref{3} and the following inequality
\[
\int_{\mathrm{M}}|\nabla u|^2d\mu=\int_{\mathrm{M}}(\lambda-q)u^2d\mu +\int_{\mathrm{M}}fud\mu.
\]

(ii) Let $\Phi\in H^2(\mathrm{M})$ so that $\Phi_{|\partial \mathrm{M}}=\varphi$ and $\|\Phi\|_{H^2(\mathrm{M})}=\|\varphi\|_{H^{3/2}(\partial \mathrm{M})}$. Then it is not hard to check that
\[
u_{q,\lambda}(\varphi)=\Phi-R_q(\lambda)((\Delta +\lambda -q)\Phi).
\]
Then \eqref{1.0} follows from \eqref{1}.
\end{proof}

 Define
\[
H_{0,\Sigma}^1(\mathrm{M})=\{ u\in H^1(\mathrm{M});\; u_{|\Sigma}=0\}.
\]

\begin{theorem}\label{theorem1.0}
Let $\mu > 0$ and $U\Subset \mathrm{Int}(\mathrm{M})$. For any $\epsilon >0$, $u\in H^2(\mathrm{M})\cap H_{0,\Sigma}^1(\mathrm{M})$ and $q\in \mathscr{Q}_\lambda$, we have
\begin{align*}
&C\|u\|_{H^1(U)}\le \epsilon ^p(1+\sqrt{\mu})\mathbf{b}_\mu\|u\|_{H^1(\mathrm{M})}
\\
&\hskip 2.0cm+\epsilon^{-1} \left( \mathbf{b}_\mu\|(\Delta -q+\mu )u\|_{L^2(\mathrm{M})}+\|\partial _\nu u\|_{L^2(\Sigma )}\right),
\end{align*}
where the constant $C>0$ and $p>0$ only depend on $n$, $\mathrm{M}$, $\kappa$, $\Sigma$ and $U$.
\end{theorem}

\begin{proof}
Let $v\in H^2(\mathrm{M}\times (0,1))$ satisfying $v_{|\Sigma\times (0,1)}=0$. We readily obtain from \cite[Propositions A.2 and A.3]{BC}
\begin{align*}
&C\|v\|_{H^1(U\times (1/4,3/4))}\le \epsilon ^p\|u\|_{H^1(\mathrm{M}\times (0,1))}
\\
&\hskip 2cm+\epsilon^{-1} \left( \|(\Delta +\partial_t^2-q )u\|_{L^2(\mathrm{M}\times (0,1))}+\|\partial _\nu v\|_{L^2(\Sigma \times (1/4,3/4))}\right),
\end{align*}
where the constants $C>0$ and $p>0$ only depend on $n$, $\mathrm{M}$, $\kappa$, $\Sigma$ and $U$.

This inequality applied to $v=e^{\sqrt{\mu}t}u$ with  $u\in H^2(\mathrm{M})\cap H_{0,\Sigma}^1(\mathrm{M})$ yields, after some technical elementary calculations,
\begin{align*}
&C\|u\|_{H^1(U)}\le \epsilon ^p(1+\sqrt{\mu})\mathbf{b}_\mu\|u\|_{H^1(\mathrm{M})}
\\
&\hskip 2cm+\epsilon^{-1} \left( \mathbf{b}_\mu\|(\Delta -q+\mu )u\|_{L^2(\mathrm{M})}+\|\partial _\nu u\|_{L^2(\Sigma )}\right).
\end{align*}
This is the expected inequality.
\end{proof}

We shall also need in the sequel the following quantitative unique continuation from boundary data.

\begin{theorem}\label{theoremB1}
Assume that $n=3$ and let $q\in \mathscr{Q}_\lambda$. For every $0<\epsilon <1$ and $u\in H_{0,\Sigma}^1(\mathrm{M} )\cap H^2(\mathrm{M})$, we have
\begin{align}
&\|u\|_{H^1(\mathrm{M})}\le C\lambda^{5/2} \mathbf{e}_\lambda \Big(\epsilon^\beta \|u\|_{H^2(\mathrm{M})}\label{B1}
\\
&\hskip 2cm +e^{c/\epsilon} \left( \|(\Delta -q+\lambda )u\|_{L^2(\mathrm{M})}+\|\partial _\nu u\|_{L^2(\Sigma )}\right)\Big),\nonumber
\end{align}
where the constant $C>0$ only depends on $\mathrm{M}$, $\kappa$, $\lambda_0$ and $\Sigma$, while the constants $c>0$ and $\beta >0$ only depend of $\mathrm{M}$, $\kappa$ and $\Sigma$.
\end{theorem}

\begin{proof}
Checking carefully the proof of  \cite[Theorem A4]{BC} and using that $H^2(\mathrm{M})$ is continuously embedded in $C^{0,1/2}(\overline{M})$ when $n=3$, for any $0<\epsilon <1$, $q\in \mathscr{Q}_\lambda$ and $u\in H_{0,\Sigma}^1(\mathrm{M} )\cap H^2(\mathrm{M})$, proceeding similarly as in the preceding proof, we get
\begin{align}
&C\|u\|_{C(\partial \mathrm{M} )}\le \lambda \epsilon^{\beta}  \|u\|_{H^2(\mathrm{M})}\label{B1.1}
\\
&\hskip 2.5cm +e^{c/\epsilon} \left( \|(\Delta -q+\lambda )u\|_{L^2(\mathrm{M})}+\|\partial _\nu u\|_{L^2(\Sigma )}\right).\nonumber
\end{align}
Here and henceforth, $C>0$ is a generic constant only depending of $\mathrm{M}$, $\kappa$, $\lambda_0$ and $\Sigma$, while $c>0$ and $\beta >0$ are generic constants only depending of $\mathrm{M}$, $\kappa$  and $\Sigma$.

Let $f=(-\Delta -\lambda+q)u$ and $\Phi\in H^1(\mathrm{M})$ be the solution of the BVP
\[
\Delta \Phi=0\; \mathrm{in}\; \mathrm{M},\quad \Phi_{|\partial M}=u_{|\partial M}.
\]
The usual a priori estimate in $H^1(\mathrm{M})$ yields
\begin{equation}\label{B3.0}
\|\Phi\|_{H^1(\mathrm{M})}\le c_0\|u\|_{H^{1/2}(\partial \mathrm{M})}.
\end{equation}
Here and in the sequel, $c_0$ is a generic constant only depending on $\mathrm{M}$.

On the other hand, it is not hard to check that
\[
u=\Phi+R_q(\lambda)(f+(\lambda-q)\Phi).
\]
In light of \eqref{B3.0}, we obtain from the proof of Lemma \ref{lemma1}
\[
\|u\|_{H^1(\mathrm{M})}\le C\left( \|u\|_{H^{1/2}(\partial \mathrm{M})}+\sqrt{\lambda}\mathbf{e}_\lambda\|f+(\lambda-q)\Phi \|_{L^2(\mathrm{M})}\right).
\]
Hence
\begin{equation}\label{B4.0}
\|u\|_{H^1(\mathrm{M})}\le C\lambda^{3/2} \mathbf{e}_\lambda\left(\|u\|_{H^{1/2}(\partial \mathrm{M})}+\|f\|_{L^2(\mathrm{M})}\right).
\end{equation}
A usual interpolation inequality yields
\[
\|u\|_{H^{1/2}(\partial \mathrm{M})}\le c_0\|u\|_{L^2(\partial \mathrm{M})}^{1/2}\|u\|_{H^1(\partial \mathrm{M})}^{1/2},
\]
 and hence
\[
\|u\|_{H^{1/2}(\partial \mathrm{M})}\le c_0\|u\|_{L^2(\partial \mathrm{M})}^{1/2}\|u\|_{H^2( \mathrm{M})}^{1/2}.
\]
Let $\rho >0$. Then an elementary convexity inequality gives
\[
c_0 \|u\|_{H^{1/2}(\partial \mathrm{M})}\le \rho\|u\|_{H^2( \mathrm{M})}+\rho^{-1}\|u\|_{L^2(\partial \mathrm{M})}.
\]
This inequality in \eqref{B4.0} yields
\[
\|u\|_{H^1(\mathrm{M})}\le C\lambda^{3/2} \mathbf{e}_\lambda\left(\rho\|u\|_{H^2( \mathrm{M})}+\rho^{-1}\|u\|_{L^2(\partial \mathrm{M})}+\|(\Delta +\lambda -q)u\|_{L^2(\mathrm{M})}\right).
\]
The last inequality and \eqref{B1.1} imply
\begin{align}
&\|u\|_{H^1(\mathrm{M})}\le C\lambda^{5/2} \mathbf{e}_\lambda \Big((\rho^{-1}\epsilon^\beta+\rho) \|u\|_{H^2(\mathrm{M})}\label{B1.2}
\\
&\hskip 2cm +\rho^{-1}e^{c/\epsilon} \left( \|(\Delta -q+\lambda )u\|_{L^2(\mathrm{M})}+\|\partial _\nu u\|_{L^2(\Sigma )}
\right).\nonumber
\end{align}
The expected inequality follows by taking $\rho=\epsilon^{\beta/2}$ in \eqref{B1.2}.
\end{proof}

\begin{corollary}\label{corollaryB1.0}
Assume that $n=3$ and let $q\in \mathscr{Q}_\lambda$. For every $0<\epsilon <1$ and $u\in H_{0,\Sigma}^1(\mathrm{M} )\cap H^2(\mathrm{M})$, we have
\begin{align}
&\|u\|_{H^{7/4}(\mathrm{M})}\le C\lambda^{5/2} \mathbf{e}_\lambda \Big(\epsilon^\beta \|u\|_{H^2(\mathrm{M})}\label{B1.2}
\\
&\hskip 2cm +e^{c/\epsilon} \left( \|(\Delta -q+\lambda )u\|_{L^2(\mathrm{M})}+\|\partial _\nu u\|_{L^2(\Sigma )}\right)\Big),\nonumber
\end{align}
where the constant $C>0$ only depends on $\mathrm{M}$, $\kappa$, $\lambda_0$ and $\Sigma$, while the constants $c>0$ and $\beta >0$ only depend on $\mathrm{M}$, $\kappa$ and $\Sigma$.\end{corollary}

\begin{proof}
Let $0<\epsilon <1$ and $u\in H_{0,\Sigma}^1(\mathrm{M} )\cap H^2(\mathrm{M})$. As $H^{7/4}(\mathrm{M})$ is an interpolation space between $H^1(\mathrm{M} )$ and $H^2(\mathrm{M} )$, we have
\begin{equation}\label{ii0.0}
C_0\|u\|_{H^{7/4}(\mathrm{M})}\le \|u\|_{H^1(\mathrm{M})}^{3/4}\|u\|_{H^2(\mathrm{M})}^{1/4}.
\end{equation}
Here and henceforth, $C_0>0$ is a generic constant only depending on $\mathrm{M}$.

Pick $\rho>0$. Combined with Young's inequality, \eqref{ii0.0} yields
\begin{equation}\label{e1.0}
C_0\|u\|_{H^{7/4}(\mathrm{M})}\le \rho^{4/3}\|u\|_{H^1(\mathrm{M})}+\rho^{-4}\|u\|_{H^2(\mathrm{M})}.
\end{equation}
Inequality \eqref{e1.0} together with \eqref{B1} yield
\begin{align}
&\|u\|_{H^{7/4}(\mathrm{M})}\le C\lambda^{5/2} \mathbf{e}_\lambda \Big((\epsilon^\beta\rho^{4/3}+\rho^{-4}) \|u\|_{H^2(\mathrm{M})}\label{B0.0}
\\
&\hskip 2.0cm +\rho^{4/3}e^{c/\epsilon} \left( \|(\Delta -q+\lambda )u\|_{L^2(\mathrm{M})}+\|\partial _\nu u\|_{L^2(\Sigma )}\right),\nonumber
\end{align}
where $C$, $c$ and $\beta$ are as in Theorem \ref{theoremB1}.

Upon modifying the constants above, the expected inequality follows by taking $\rho=\epsilon^{-3\beta/16}$ in \eqref{B0.0}.
\end{proof}

We denote in the sequel by $\chi$ the characteristic function of $\mathrm{M}_0$, that we consider as a function defined on $\mathrm{M}$.

We will use in the next section the following consequence of Corollary \ref{corollaryB1.0}, applied with $\mathrm{M}$ substituted by $\mathrm{M}_1$ and $\Sigma$ substituted by $\Gamma$, together with Lemma \ref{lemma1}.

\begin{corollary}\label{corollary1.0}
Suppose that $n=3$ and let $q\in \mathscr{Q}_\lambda$. For any $\epsilon >0$, $f\in L^2(\mathrm{M})$ and $u=R_q(\lambda)(\chi f)$, we have
\begin{equation}\label{e0}
C\|u\|_{H^{7/4}(\mathrm{M}_1)}\le   \lambda^{7/2}\mathbf{e}_\lambda\left(\epsilon^\beta \|f\|_{L^2(\mathrm{M}_0)} + e^{c/\epsilon} \|\partial_\nu u\|_{L^2(\Gamma )}\right),
\end{equation}
where the constant $C>0$ only depends on $\mathrm{M}$, $\mathrm{M}_0$, $\kappa$, $\lambda_0$ and $\Gamma$, while the constants $c>0$ and $\beta >0$ only depend of $\mathrm{M}$, $\mathrm{M}_0$, $\kappa$ and $\Gamma$.
\end{corollary}

\section{Quantitative Runge approximation}\label{section3}

In all of this section $n=3$. For $q\in \mathscr{Q}_\lambda$, define
\begin{align*}
&\mathscr{S}_{q,\lambda}^0=\{ u\in H^2(\mathrm{M}_0);\; (\Delta+\lambda -q)u=0\},
\\
&\mathscr{S}_{q,\lambda}^1=\{ u\in H^2(\mathrm{M});\; (\Delta+\lambda -q)u=0,\; u_{|\partial M}\in H_\Gamma^{3/2}(\partial \mathrm{M})\}.
\end{align*}

\begin{theorem}\label{theorem2}
We find two constants  $C>0$ and $c>0$, only depending on $\mathrm{M}$,  $\mathrm{M}_0$, $\kappa$, $\lambda_0$ and  $\Gamma$, and a constant $\beta>0$, only depending on $\mathrm{M}$,  $\mathrm{M}_0$, $\kappa$ and  $\Gamma$, such that, for any $q\in \mathscr{Q}_\lambda$, $0<\epsilon <1$ and $u\in \mathscr{S}_{q,\lambda}^0$, there exists $v\in \mathscr{S}_{q,\lambda}^1$ for which the following inequalities hold.
\begin{align}
C\|u-v_{|\mathrm{M}_0}\|_{L^2(\mathrm{M}_0)}&\le \lambda^{17/4}\mathbf{e}_\lambda^{7/4} \epsilon^\beta\|u\|_{H^2(\mathrm{M}_0)},\label{ra}
\\
 \|v\|_{H^2(\mathrm{M})}&\le \lambda^2\mathbf{e}_\lambda e^{c/\epsilon}\|u\|_{L^2(\mathrm{M}_0)}.\label{ra2}
\end{align}
\end{theorem}

\begin{proof}
Let $\Gamma_0\Subset \Gamma$ and  $\psi \in C^\infty (\partial M)$ satisfying $0\le \psi \le 1$, $\psi =1$ in $\overline{\Gamma_0}$ and $\psi=0$ in $\partial \mathrm{M}\setminus \overline{\Gamma}$.

Let $E:=H^{3/2}(\partial \mathrm{M})$ and consider the closed subset of $E$ given by
\[
F=\{f\in E;\; \psi f=0\}.
\]
Let $\pi:E\rightarrow E/F$ be the quotient map and set for $\varphi\in E/F$ 
\[
\dot{\varphi}=\{f\in E;\; \pi(f)=\varphi\}. 
\]
We equip $E/F$ with its usual quotient norm
\[
\|\varphi\|_{E/F}=\inf\{\|f\|_E;\; f\in \dot{\varphi}\}.
\]
Recall that if $F^\bot:=\{f^\ast\in E^\ast,\; f^\ast_{|F}=0\}$, then the linear map $(E/F)^\ast\rightarrow F^\bot:f^\ast \mapsto f^\ast \circ \pi$ defines an isometric isomorphism. If necessary, we henceforth identify $(E/F)^\ast$ with $F^\bot$.

Fix $q\in \mathscr{Q}_\lambda$ and denote by $H$ the closure of the subspace $\mathscr{S}_{q,\lambda}^0$ in $L^2(\mathrm{M}_0)$, that we  identify with its dual.

Let $\varphi \in E/F$ and $f\in \dot{\varphi}$. Let $u=u_{q,\lambda}(\psi\overline{f})\in H^2(M)$. That is, $u$ is the unique solution of the BVP
\begin{equation}\label{bvp1}
(\Delta +\lambda -q)u=0\; \mathrm{in}\; \mathrm{M},\quad u_{|\partial \mathrm{M}}=\psi\overline{f}.
\end{equation}
If $f_1,f_2\in \dot{\varphi}$, then $\psi f_1=\psi f_2$. In consequence, $u_{q,\lambda}(\psi\overline{f})$ does not depend on $f\in \dot{\varphi}$.

According to Riesz's representation theorem, we have the following canonical isometric isomorphism
\[
\mathscr{I}:(E/F)^\ast \rightarrow E/F: \langle f^\ast,f\rangle =(\mathscr{I}(f^\ast)|f)_{E/F}, 
\]
where $(\cdot |\cdot)_{E/F}$ denotes the scalar product of $E/F$.

Consider the operator
\[
T: E/F\rightarrow H:\varphi\mapsto T\varphi :=u_{|\mathrm{M}_0},
\]
where $u=u_{q,\lambda}(\psi\overline{f})\in H^2(M)$ with $f\in \dot{\varphi}$ arbitrary chosen.

For simplicity, the Riemannian measure on $\mathrm{M}$ and $\mathrm{M}_0$ (resp. $\partial \mathrm{M}$ and $\partial \mathrm{M}_0$) will denoted by $\mu$ (resp. $s$).

Pick $v\in L^2(\mathrm{M}_0)$, $\varphi\in E/F$, $f\in \dot{\varphi}$  and let $w=R_q(\lambda)(v\chi)$. Applying Green's formula, we find
\[
(T\varphi|v)_{L^2(\mathrm{M}_0)}=\int_{\mathrm{M_0}}u\overline{v}d\mu =\int_{\mathrm{M}}u\overline{v}\chi d\mu= \int_{\mathrm{M}}u(\Delta +\lambda -q)\overline{w}d\mu
=\int_{\partial \mathrm{M}}\psi \overline{f}\partial_\nu \overline{w}ds
\]
If $f\in F$ then 
\begin{equation}\label{0.0}
\int_{\partial \mathrm{M}}\psi \overline{f}\partial_\nu \overline{w}ds=0.
\end{equation}
That is $\psi \partial_\nu \overline{w}$ belongs to $F^\bot$ as therefore it can be considered as an element of $(E/F)^\ast$.  We can then define $T':H\rightarrow (E/F)^\ast$ by 
\[
\langle \varphi, T'v\rangle :=\int_{\partial \mathrm{M}}\psi \overline{f}\partial_\nu \overline{w}ds.
\]
Note that, in light of \eqref{0.0}, the right hand side of inequality above is independent of $f\in \dot{\varphi}$. In consequence, the adjoint of $T$ is given as follows
\[
T^\ast: H\rightarrow E/F:v\mapsto T^\ast v=\overline{\mathscr{I}(T'v)}.
\]

Let us check that $T$ is injective. To this end, we see that if $u_{|\mathrm{M}_0}=T\varphi=0$ then $u$ satisfies
\[
(\Delta +\lambda -q)u=0\; \mathrm{in}\; \mathrm{M},\quad u_{|\mathrm{M}_0}=0.
\]
Hence $u=0$ according to the unique continuation property, implying that $\psi \overline{f}=0$. That is, $f\in F=\dot{0}$ and therefore $T$ is injective.

Next, let us show that $T$ has dense range. If $T^\ast v=0$, then $w=R_q(\lambda)(v\chi)$ satisfies
\[
\int_{\partial \mathrm{M}}\psi f \partial_\nu wds=0,\quad f\in E\; (=\pi^{-1}(E/F)).
\]
Hence $\psi \partial_\nu w=0$ and therefore
\[
(\Delta +\lambda -q)w=0\; \mathrm{in}\; \mathrm{M}_1,\quad w_{|\partial\mathrm{M} }=0,\; \partial_\nu w_{|\Gamma_0}=0.
\]
Then $w=0$ in $M_1$, again according to the unique continuation property.

Taking $u\in \mathscr{S}_{q,\lambda}^0$ and applying Green's formula, we obtain
\[
(u|v)_{L^2(\mathrm{M}_0)}= (u|(\Delta+\lambda -q)w)_{L^2(\mathrm{M}_0)}=((\Delta+\lambda -q)u|w)_{L^2(\mathrm{M}_0)}=0.
\]
Whence, $v\in ( \mathscr{S}_{q,\lambda}^0)^\bot=H^\bot$ implying that $v=0$. Hence, $T$ has dense range.

We claim that $T$ is compact. Indeed, if $(\varphi_j)$ is a bounded sequence in $E/F$ then there exists a bounded sequence $(f_j)$ of $E$ such that $f_j\in \dot{\varphi}$ for each $j$. In consequence,  $(u_{q,\lambda}(\psi\overline{f_j}))$ is bounded sequence in $H^2(\mathrm{M})$ according to the a priori estimate in $H^2(\mathrm{M})$. Subtracting if necessary a subsequence, we assume that $(u_{q,\lambda}(\psi \overline{f_j}))$ converges weakly in $H^2(\mathrm{M})$ and strongly in $L^2(\mathrm{M})$ to $u\in H^2(\mathrm{M})$. Since for all $j$, $u_{q,\lambda}(\psi\overline{f_j})_{|\mathrm{M}_0}\in \mathscr{S}_{q,\lambda}^0$ for all $j$, we conclude that $u\in  H$.

 Summing up, we see that $T^\ast T: E/F\rightarrow E/F$ is compact self-adjoint and positive definite operator. It is therefore diagonalizable. Hence, there exists a sequence of positive numbers $(\mu_j)$ and an orthonormal basis $(\psi_j)$ of $E/F$ so that
\[
T^\ast T\psi_j=\mu_j \psi_j.
\]
Define $\tau_j=\mu_j^{1/2}$ and $u_j=\mu_j^{-1/2}T\psi_j\in H$. Then we have
\begin{align*}
(u_j|u_k)_{L^2(\mathrm{M}_0)}&=\mu_j^{-1/2}\mu_k^{-1/2}(T\psi_j|T\psi_k)_{L^2(\mathrm{M}_0)}
\\
&=\mu_j^{-1/2}\mu_k^{-1/2}(T^\ast T\psi_j|\psi_k)_{E/F}=\delta_{jk}.
\end{align*}
Next, let $u\in H$ so that $(u|u_j)_{L^2(\mathrm{M}_0)}=0$ for all $j$. Then, $(u|T\psi_j)_{L^2(\mathrm{M}_0)}=0$ for all $j$ and hence $(u|T\varphi)_{L^2(\mathrm{M}_0)}=0$ for all $\varphi\in E/F$. As $T$ has a dense range in $H$, we derive that $(u|v)_{L^2(\mathrm{M}_0)}=0$, for every $v\in H$. Taking $v=u$ yields $u=0$. Therefore $(u_j)$ is an orthonormal basis of $H$.

We observe that, since $T^\ast u_j=\tau_j\psi_j$, we have
\begin{equation}\label{5}
\|T^\ast u_j\|_{E/F}=\tau_j.
\end{equation}

Let $t>0$, $u=\sum_j a_ju_j\in H$ and set
\[
\varphi_t=\sum_{N_t^1} \tau_j^{-1} a_j\psi_j,
\]
where $N_t^1=\{j;\; \tau_j>t\}$.

We have
\[
\|\varphi_t\|_{E/F}^2=\sum_{N_t^1} \tau_j^{-2} |a_j|^2\le t^{-2}\sum_j |a_j|^2=t^{-2}\|u\|_{L^2(\mathrm{M}_0)}^2.
\]
That is, the following inequality holds
\begin{equation}\label{4}
\|\varphi_t\|_{E/F}\le t^{-1}\|u\|_{L^2(\mathrm{M}_0)}.
\end{equation}

Let $N_t^0=\{j;\; \tau_j\le t\}$ and
\[
v_t =\sum_{j\in N_t^0} a_ju_j.
\]
Since $u-T\varphi_t=v_t$ and $T\varphi_t\bot v_t$, we find
\[
\|v_t\|_{L^2(\mathrm{M}_0)}^2=(v_t|v_t)_{L^2(\mathrm{M}_0)}=(v_t+Tg_t|v_t)_{L^2(\mathrm{M}_0)}=(u|v_t)_{L^2(\mathrm{M}_0)}.
\]
Let $w_t=R_q(\lambda)(\chi v_t)$. Then we have
\[
\|v_t\|_{L^2(\mathrm{M}_0)}^2=(u|(\Delta +\lambda -q)w_t)_{L^2(\mathrm{M}_0)}.
\]
Using $(\Delta +\lambda -q)u=0$ and applying Green's formula, we obtain
\begin{align*}
\|v_t\|_{L^2(\mathrm{M}_0)}^2&=(u|\partial_\nu w_t)_{L^2(\partial \mathrm{M}_0)}-(\partial_\nu u|w_t)_{L^2(\partial \mathrm{M}_0)}
\\
&\le \|u\|_{L^2(\partial \mathrm{M}_0)} \|\partial_\nu w_t\|_{L^2(\partial \mathrm{M}_0)}+\|\partial_\nu u\|_{L^2(\partial \mathrm{M}_0)}\|w_t\|_{L^2(\partial \mathrm{M}_0)}.
\end{align*}

Noting that the trace mapping
\[
h\in H^{7/4}(\mathrm{M}_1)\mapsto \left(h_{|\partial \mathrm{M}_0},\partial_\nu h_{|\partial \mathrm{M}_0}\right)\in L^2(\partial \mathrm{M}_0)\times L^2(\partial \mathrm{M}_0)
\]
 is bounded, we find a constant $c_0>0$ only depending on $\mathrm{M}_1$  so that
\begin{equation}\label{6}
\|v_t\|_{L^2(\mathrm{M}_0)}^2\le c_0\|u\|_{H^2(\mathrm{M}_0)}\|w_t\|_{H^{7/4}(\mathrm{M}_1)}.
\end{equation}
Pick $ 0<\epsilon <1$ arbitrary. It follows from Corollary \ref{corollary1.0}, in which $\Gamma$ is replaced by $\Gamma_0$, that
\[
\|w_t\|_{H^{7/4}(\mathrm{M}_1)}\le C\lambda^{7/2}\mathbf{e}_\lambda \left[\epsilon^\beta \|v_t\|_{L^2(\mathrm{M}_0)}+ e^{c/\epsilon}\|\partial_\nu w_t\|_{L^2(\Gamma_0)}\right].
\]
Here and henceforth, $C>0$ and $c>0$ are generic constants only depending on $\mathrm{M}$,  $\mathrm{M}_0$, $\kappa$, $\lambda_0$ and  $\Gamma$, and a constant $\beta>0$ is generic constant only depending on $\mathrm{M}$,  $\mathrm{M}_0$, $\kappa$ and  $\Gamma$. In particular, we have
\begin{equation}\label{0.1}
\|w_t\|_{H^{7/4}(\mathrm{M}_1)}\le C\lambda^{7/2}\mathbf{e}_\lambda \left[\epsilon^\beta \|v_t\|_{L^2(\mathrm{M}_0)}+ e^{c/\epsilon}\|\psi \partial_\nu w_t\|_{L^2(\partial \mathrm{M})}\right].
\end{equation}

It follows from  the definition of $T'$ and \eqref{5} that
\begin{equation}\label{0.2}
\|\psi \partial_\nu w_t\|_{H^{-3/2}(\partial M)}= \|T^\ast v_t\|_{E/F}\le t\|v_t\|_{L^2(\mathrm{M}_0)}.
\end{equation}
On the other hand, the interpolation inequality in \cite[Theorem 7.7]{LM} gives
\begin{align*}
\|\psi \partial_\nu w_t\|_{L^2(\partial M)}&\le C\|\psi \partial_\nu w_t\|_{H^{-3/2}(\partial M)}^{1/4}\|\psi \partial_\nu w_t\|_{H^{1/2}(\partial M)}^{3/4}
\\
&\le C\|\psi \partial_\nu w_t\|_{H^{-3/2}(\partial M)}^{1/4}\|w_t\|_{H^2(\mathrm{M})}^{3/4}.
\end{align*}
This inequality, combined with \eqref{1} and \eqref{0.2}, gives
\begin{equation}\label{0.3}
\|\psi \partial_\nu w_t\|_{L^2(\partial M)}\le C[\lambda \mathbf{e}_\lambda]^{3/4}t^{1/4}\|v_t\|_{L^2(\mathrm{M}_0)}.
\end{equation}
Hence, \eqref{0.3} in \eqref{0.1} yields
\[
\|w_t\|_{H^{7/4}(\mathrm{M}_1)}\le C\lambda^{17/4}\mathbf{e}_\lambda^{7/4}\left[\epsilon ^\beta + t^{1/4}e^{c/\epsilon}\right]\|v_t\|_{L^2(\mathrm{M}_0)}.
\]
This and \eqref{6} yield
\[
\|v_t\|_{L^2(\mathrm{M}_0)}\le C\lambda^{17/4}\mathbf{e}_\lambda^{7/4}\|u\|_{H^2(\mathrm{M}_0)}\left[\epsilon ^\beta + t^{1/4}e^{c/\epsilon}\right].
\]
In this inequality taking $t^{1/4}=\epsilon^\beta e^{-c/\epsilon}$, we obtain
\begin{equation}\label{9}
\|v_t\|_{L^2(\mathrm{M}_0)}\le C\lambda^{17/4}\mathbf{e}_\lambda^{7/4} \epsilon^\beta \|u\|_{H^2(\mathrm{M}_0)},
\end{equation}
and in light of \eqref{4} we obtain
\begin{equation}\label{10}
\|\varphi_t\|_{E/F}\le e^{c/\epsilon}\|u\|_{L^2(\mathrm{M}_0)}.
\end{equation}
Let $v=u_{q,\lambda}(\psi \overline{f_t})$, where $f_t\in \dot{\varphi}_t$ is chosen arbitrarily. Then $v\in \mathscr{S}_{q,\lambda}^1$ and, since $u-T\varphi_t=u-v_{|\mathrm{M}_0}=v_t$,  we obtain \eqref{ra}-\eqref{ra2} from \eqref{9}-\eqref{10} and Lemma \ref{lemma1} (ii).
\end{proof}

\section{Preliminary inequality}\label{section4}

We assume in this section that $n=3$. Henceforth,  $C>0$ is a generic constant only depending on $\mathrm{M}$, $\mathrm{M}_0$, $\kappa$, $\lambda_0$, $\Sigma$ and $\Gamma$,  $c>0$ is a generic constant only depending on $\mathrm{M}$, $\mathrm{M}_0$, $\kappa$, $\lambda_0$, and $\Gamma$, while $\beta>0$ is a generic constant only depending on $\mathrm{M}$,  $\mathrm{M}_0$, $\kappa$ and  $\Gamma$

Let $q_1,q_2\in \mathscr{Q}_\lambda$ satisfying $q_1=q_2$ in $\mathrm{M}_1$. Pick $0<\epsilon <1$ and $u_j\in \mathscr{S}_{q_j,\lambda}^0$, $j=1,2$. From Theorem \ref{theorem2}, there exist $v_j\in \mathscr{S}_{q_j,\lambda}^1$, $j=1,2$, so that
\begin{equation}\label{2.2}
C\|u_j-{v_j}_{|\mathrm{M}_0}\|_{L^2(\mathrm{M}_0)}\le \lambda^{17/4}\mathbf{e}_\lambda^{7/4}\epsilon^{\beta} \|u_j\|_{H^2(\mathrm{M}_0)},
\end{equation}
and
\begin{equation}\label{2.3}
\|v_j\|_{H^2(\mathrm{M})}\le \lambda^2\mathbf{e}_\lambda e^{c/\epsilon}\|u_j\|_{L^2(\mathrm{M}_0)}.
\end{equation}

In light of \eqref{2.2}, using the identity
\[
u_1u_2=(u_1-v_1)u_2+(v_1-u_1)(u_2-v_2)+u_1(u_2-v_2)+v_1v_2,
\]
and  the fact that $\epsilon < 1$, we obtain

\begin{align}
&C\left|\int_{\mathrm{M}_0}(q_1-q_2)u_1u_2d\mu\right|\le\lambda^{17/2}\mathbf{e}_\lambda^{7/2}\epsilon^\beta \|u_1\|_{H^2(\mathrm{M}_0)}\|u_2\|_{H^2(\mathrm{M}_0)} \label{2.4}
\\
&\hskip7cm+\left|\int_{\mathrm{M}_0}(q_1-q_2)v_1v_2d\mu\right|.\nonumber
\end{align}

Let $\tilde{v}_2\in H^2(M)$ be the solution of the BVP
\[
(\Delta +\lambda -q_1)\tilde{v}_2=0\; \mathrm{in}\; \mathrm{M},\quad \tilde{v}_2{_{|\partial M}}= v_2{_{|\partial M}},
\]
and $v=v_2-\tilde{v}_2=R_{q_1}(\lambda)(\chi (q_1-q_2)v_2)$.

Pick $\psi \in C_0^\infty (\mathrm{M}'_0)$ so that $\psi=1$ in a neighborhood of $\mathrm{M}_0$. Then $w=\psi v$ satisfies
\[
(\Delta +\lambda -q_1)w= \chi(q_2-q_1)v_2+[\Delta,\psi]v.
\]
Applying Green's formula, we obtain
\[
0=\int_M\chi(q_2-q_1)v_2v_1dx+\int_M[\Delta,\psi]vv_1dx.
\]
Since $\mathrm{supp}[\Delta,\psi]v\subset U\Subset \mathrm{M}_1$, we get
\[
\left|\int_M(q_2-q_1)v_2v_1dx\right|\le C_0\|v\|_{H^1(U)}\|v_1\|_{L^2(\mathrm{M})}.
\]
This inequality, combined with Theorem \ref{theorem1.0}, in which $\mathrm{M}$ is substituted by $\mathrm{M}_1$, and Lemma \ref{lemma1}, yields
\[
\left|\int_M(q_2-q_1)v_2v_1dx\right|\le C\lambda\mathbf{b}_\lambda\left[\mathbf{e}_\lambda\rho^p\|v_2\|_{L^2(\mathrm{M})}+ \rho^{-1}\| \partial_\nu v\|_{L^2(\Sigma)}\right]\|v_1\|_{L^2(\mathrm{M})},
\]
where $\rho>0$ is taken arbitrary and $p$ is as in Theorem \ref{theorem1.0}.

Since
\begin{align*}
\| \partial_\nu v\|_{L^2(\Sigma)}&=\| (\Lambda_{q_1,\lambda}^0-\Lambda_{q_2,\lambda}^0)(v_2{_{|\partial M}})\|_{L^2(\Sigma)}
\\
&\le \| \Lambda_{q_1,\lambda}^0-\Lambda_{q_2,\lambda}^0\| \|v_2\|_{H^{3/2}(\Sigma)}
\\
&\le \lambda^2\mathbf{e}_\lambda e^{c/\epsilon} \| \Lambda_{q_1,\lambda}^0-\Lambda_{q_2,\lambda}^0\| \|u_2\|_{L^2(\mathrm{M}_0)},
\end{align*}
we derive
\begin{align*}
&C\left|\int_M(q_2-q_1)v_2v_1dx\right|
\\
&\hskip 1cm \le \lambda^5\mathbf{e}_\lambda^2\mathbf{b}_\lambda e^{c/\epsilon} \left[\rho^p+\rho^{-1}  \| \Lambda_{q_1,\lambda}^0-\Lambda_{q_2,\lambda}^0\|\right]\|u_1\|_{L^2(\mathrm{M}_0)}\|u_2\|_{L^2(\mathrm{M}_0)}.
\end{align*}

Taking in this inequality $\rho=\| \Lambda_{q_1}^0-\Lambda_{q_2}^0\|^{1/(p+1)}$, we get
\[
C\left|\int_M(q_2-q_1)v_2v_1dx\right|\le \lambda^5\mathbf{e}_\lambda^2\mathbf{b}_\lambda e^{c/\epsilon}  \| \Lambda_{q_1,\lambda}^0-\Lambda_{q_2,\lambda}^0\|^\theta \|u_1\|_{L^2(\mathrm{M}_0)}\|u_2\|_{L^2(\mathrm{M}_0)},
\]
where $\theta=p/(p+1)$. This inequality in  \eqref{2.4} yields the following result, where 
\[
\mathbf{n}_\lambda=\max\left( \lambda^{7/2}\mathbf{e}_\lambda^{3/2},\mathbf{b}_\lambda\right)\lambda^5\mathbf{e}_\lambda^2=\lambda^{-2}\mathbf{m}_\lambda.
\]

\begin{proposition}\label{proposition2.1}
Let $q_j\in \mathscr{Q}_\lambda$, $j=1,2$ satisfying $q_1=q_2$ in $\mathrm{M}_1$. For each $0<\epsilon<1$ and $u_j\in \mathscr{S}_{q_j,\lambda}^0$, $j=1,2$, we have
\begin{align}
&C\mathbf{n}_\lambda^{-1}\left|\int_{\mathrm{M}_0}(q_1-q_2)u_1u_2d\mu\right|\le \epsilon^\beta \|u_1\|_{H^2(\mathrm{M}_0)}\|u_2\|_{H^2(\mathrm{M}_0)}\label{2.4.1}
\\
&\hskip 4.5cm   +e^{c/\epsilon}  \| \Lambda_{q_1,\lambda}^0-\Lambda_{q_2,\lambda}^0\|^\theta \|u_1\|_{L^2(\mathrm{M}_0)}\|u_2\|_{L^2(\mathrm{M}_0)},\nonumber
\end{align}
where the constant $C>0$ only depends on $\mathrm{M}$, $\mathrm{M}_0$, $\Gamma$, $\Sigma$, $\lambda_0$ and $\kappa$, the constant $c>0$ only depends on $\mathrm{M}$, $\mathrm{M}_0$, $\Gamma$, $\lambda_0$ and $\kappa$, while the constant $0<\theta<1$ only depends on $\mathrm{M}$, $\mathrm{M}_0$, $\Sigma$ and $\kappa$, and the constant $\beta>0$, only depends on $\mathrm{M}$,  $\mathrm{M}_0$, $\kappa$ and  $\Gamma$.
\end{proposition}

\section{A consequence of the preliminary inequality}\label{section5}

Again, in this section we assume that $n=3$. Recall that a (smooth) compact connected manifold $(\mathrm{N},g)$ with boundary is simple if $\partial \mathrm{N}$ is strictly convex, and for any point $x\in \mathrm{N}$ the exponential map $\mathrm{exp}_x$ is a diffeomorphism from some closed neighborhood of $0$ in $T_x\mathrm{N}$ onto $\mathrm{N}$.

A (smooth) connected compact manifold $(\mathrm{N},g)$ with boundary is called  admissible if it is conformal to a submanifold with boundary of $\mathbb{R}\times (\mathrm{N}',g')$ where $(\mathrm{N}',g')$ is a simple $n-1$ dimensional manifold.

We refer to \cite{DKSU} for more detailed definitions, and we assume in the rest of this section that $(\mathrm{M}_0,g)$ is admissible.

It is worth noticing that \eqref{2.4.1} still holds for fixed $q_j\in L^\infty (\mathrm{M},\mathbb{R})$, $j=1,2$ and $\lambda \in \mathbb{\rho}(A_{q_1})\cap \mathbb{\rho}(A_{q_2})$, where the constants $C$, $c$, $\beta$ and $\theta$ may also depend on $q_j$. Then, in light of \cite[Subsection 6.1]{DKSU}, we easily derive from  \eqref{2.4.1} the following result

\begin{theorem}\label{theorem6.1}
Let $q_1$, $q_2\in L^\infty (\mathrm{M},\mathbb{R})$  satisfying $q_1=q_2$ in $\mathrm{M}_1$, and let $\lambda \in \mathbb{\rho}(A_{q_1})\cap \mathbb{\rho}(A_{q_2})$. If $\Lambda_{q_1,\lambda}^0=\Lambda_{q_2,\lambda}^0$ then $q_1=q_2$.
\end{theorem}

Next, define
\[
\mathfrak{C}=\{ \mathfrak{c} \in C^\infty(\mathrm{M},\mathbb{R});\; \mathfrak{c} >0 \}.
\]

For $\mathfrak{c} \in \mathfrak{C}$, consider the partial Dirichlet-to-Neumann map $\Pi_\mathfrak{c}$ given as follows
\[
\Pi_\mathfrak{c}: \varphi\in H_\Gamma^{3/2}(\partial \mathrm{M})\mapsto \partial_\nu u_\mathfrak{c}(\varphi)_{|\Sigma},
\]
where $u_\mathfrak{c}(\varphi)\in H^2(\mathrm{M})$ is the unique solution of the BVP
\[
\Delta_{\mathfrak{c} g}u=0\; \mathrm{in}\; \mathrm{M},\quad u_{|\partial \mathrm{M}}=\varphi.
\]

Let $\mathbf{1}$ denotes the function identically equal to $1$ in $\mathrm{M}$. Pick $\mathfrak{c}\in \mathfrak{C}$ satisfying $\Pi_{\mathfrak{c}}=\Pi_{\mathbf{1}}$ and $\mathfrak{c}=1$ in $\mathrm{M}_1$. If $q_\mathfrak{c} =\Delta_g(\mathfrak{c}^{(n-2)/4})/\mathfrak{c}^{(n+2)/4}$ then we obtain from the computations in \cite[ proof of Theorem 1.8]{DKSU} that
\[
\Lambda_{q_\mathfrak{c},0}^0=\Lambda_{0,0}^0.
\]
Hence $q_\mathfrak{c}=0$ by Theorem \ref{theorem6.1}. Therefore, $w=\mathfrak{c}^{(n-2)/4}$ is the solution of the BVP
\[
\Delta_gw=0\; \mathrm{in}\; \mathrm{M},\quad w_{|\partial \mathrm{M}}=1.
\]
By uniqueness of solutions of this BVP, we have $w=\mathbf{1}$ and then $\mathfrak{c}=\mathbf{1}$.

In other words, we proved the following result.

\begin{corollary}\label{corollary6.2.0}
Let $\mathfrak{c}\in \mathfrak{C}$ satisfying $\mathfrak{c}=\mathbf{1}$ in $\mathrm{M}_1$ and $\Pi_{\mathfrak{c}}=\Pi_{\mathbf{1}}$. Then $\mathfrak{c}=\mathbf{1}$.
\end{corollary}

\section{Proof of Theorems \ref{mt} and \ref{mt1}}\label{section6}

In this section $\mathrm{M}$ is $C^{1,1}$ bounded domain of $\mathbb{R}^n$, $n\ge 3$, endowed with the round metric. We shall use the complex geometric optics solutions given by the following proposition.

\begin{proposition}\label{proposition3.1}
Let $q\in \mathscr{Q}_\lambda$. There exists a constant $\varpi\ge 1$ only depending on $n$, $\mathrm{M}_0$ and $\kappa$ such that, for each $\xi \in \mathbb{C}^n$ satisfying $\xi\cdot\xi=\lambda$ and $|\Im \xi|>2\varpi$, we find $u_\xi=e^{-ix\cdot \xi}(1+w_\xi)\in \mathscr{S}_{q,\lambda}^0$ and the following inequalities hold
\begin{align}
&\|w_\xi\|_{L^2(\mathrm{M}_0)}\le C|\Im \xi|^{-1}, \label{3.1}
\\
&\|u_\xi\|_{H^2(\mathrm{M}_0)}\le C\lambda e^{\varkappa |\Im \xi|},\label{3.2}
\end{align}
where the constant $C>0$  only depends on $n$, $\mathrm{M}_0$, $\lambda_0$ and $\kappa$ and the constant $\varkappa >0$ only depends on $n$, $\mathrm{M}_0$.
\end{proposition}

Note that $w_\xi$ in Proposition \ref{proposition3.1} must satisfy the same equation as in the case $\lambda=0$ (see \cite[(2.5)]{Ch2}). Therefore the proof of Proposition \ref{proposition3.1} is obtained by modifying slightly  those of \cite[Proposition 2.6]{Ch2} and \cite[Lemme 2.15]{Ch2}.

Pick $\eta \in \mathbb{R}^n$ and choose $\eta_1,\eta_2\in \mathbb{R}^n\setminus\{0\}$ in such a way that $\eta_1\bot\eta_2$, $\eta_j\bot \eta$, $j=1,2$, $|\eta_1|^2>4\varpi^2+\lambda $ and
\[
|\eta_2|^2=|\eta |^2/4+|\eta_1|^2-\lambda .
\]
Here and henceforth the constant $\varpi$ is as in Proposition \ref{proposition3.1}. Then set
\[
\xi_1=(\eta /2+\eta_1)+i\eta_2,\quad \xi_2=(\eta /2-\eta_1)-i\eta_2.
\]
We check that we have
\begin{equation}\label{3.3}
|\Im \xi_j|=|\eta_2|>2\varpi,\quad \xi_j\cdot \xi_j=\lambda,\quad j=1,2,\quad \xi_1+\xi_2=\eta.
\end{equation}

In the sequel we assume that $|\eta_1|^2=\lambda+\tau^2$, where $\tau >2\varpi$. In this case we have $|\Im \xi_j|\ge \tau>2\varpi$ and $|\Im \xi_j|\le \tau +|\eta |/2$, $j=1,2$.

Let $q_j\in \mathscr{Q}_\lambda$, $j=1,2$, so that $q_1=q_2$ in $\mathrm{M}_1$. Set $u_j=u_{\xi_j}$ and $w_j=w_{\xi_j}$, $j=1,2$, where $u_{\xi_j}$ and $w_{\xi_j}$ are as in Proposition \ref{proposition3.1} when $q=q_j$.

Define $q=q_1-q_2$ extended by $0$ outside $\mathrm{M}$. Since
\[
u_1u_2=e^{-i\eta\cdot x}+\varrho,\quad \varrho=e^{-i\eta\cdot x}(w_1+w_2+w_1w_2),
\]
we find
\[
\hat{q}(\eta)=\int_{\mathrm{M}_0}qu_1u_2dx-\int_{\mathrm{M}_0}q\varrho dx,
\]
where $\hat{q}$ denotes the Fourier transform of $q$. Hence
\begin{equation}\label{3.4}
|\hat{q}(\eta)|\le \left| \int_{\mathrm{M}_0}qu_1u_2dx\right|+C_0\tau^{-1}
\end{equation}
by \eqref{3.1}, where the constant $C_0>0$ only depends on $n$, $\mathrm{M}_0$ and $\kappa$.

On the other hand, it follows from \eqref{3.2} that
\begin{equation}\label{3.5}
\|u_1\|_{H^2(\mathrm{M}_0)}\|u_2\|_{H^2(\mathrm{M}_0)}\le C_1\lambda^2 e^{\varkappa (|\eta|+2\tau)},
\end{equation}
where the constant $C_1>0$ only depends on $n$, $\mathrm{M}_0$, $\lambda_0$ and $\kappa$, and $2\varkappa$ was substituted by $\varkappa$.

\subsection{Proof of Theorem \ref{mt}}

Using inequalities \eqref{3.4} and \eqref{3.5} in \eqref{2.4.1}, and the fact that $\mathfrak{m}_\lambda=\lambda^2\mathfrak{n}_\lambda$, we obtain
\[
C\mathbf{m}_\lambda^{-1}|\hat{q}(\eta)|\le \tau^{-1}+\epsilon^\beta e^{\varkappa (|\eta|+2\tau)} +e^{c/\epsilon}e^{\varkappa (|\eta|+2\tau)}  \| \Lambda_{q_1,\lambda}^0-\Lambda_{q_2,\lambda}^0\|^\theta .
\]
Here and henceforth, $C>0$ denotes a generic constant only depending on $\mathrm{M}$, $\mathrm{M}_0$, $\Gamma$, $\Sigma$, $\lambda_0$ and $\kappa$. While $c$, $\beta$ and $\theta$ are as in \eqref{2.4.1}, and $\varkappa$ is the constant appearing in \eqref{3.5}.

Pick $s>0$. Therefore, we have for $|\eta|\le s$
\[
C\mathbf{m}_\lambda^{-1}|\hat{q}(\eta)|\le \tau^{-1}+\epsilon^\beta e^{\varkappa (s+2\tau)} + e^{c/\epsilon}e^{\varkappa (s+2\tau)}  \| \Lambda_{q_1,\lambda}^0-\Lambda_{q_2,\lambda}^0\|^\theta .
\]

Next, upon substituting $2c$ by $c$, $2\beta$ by $\beta$, $4\varkappa$ by $\varkappa$ and $2\theta$ by $\theta$, we get
\begin{align}
&C\mathbf{m}_\lambda^{-2}\int_{B(0,s)}|\hat{q}(\eta)|^2d\eta \label{3.6}
\\
&\hskip 2cm \le s^n\tau^{-2}+s^n\epsilon^\beta e^{\varkappa (s+\tau)} + s^ne^{c/\epsilon}e^{\varkappa (s+\tau)}  \| \Lambda_{q_1,\lambda}^0-\Lambda_{q_2,\lambda}^0\|^\theta.\nonumber
\end{align}

On the other hand, we have
\begin{equation}\label{3.7}
\int_{|\eta|\ge s}(1+|\eta|^2)^{-1}|\hat{q}(\eta)|^2d\eta \le s^{-2}\|q\|_{L^2(\mathbb{R}^n)}\le C's^{-2},
\end{equation}
where the constant $C'>0$ only depends on $\kappa$ and $|\mathrm{M}|$.

Putting together \eqref{3.6} and \eqref{3.7}, we obtain
\begin{align*}
&C\mathbf{m}_\lambda^{-2}\|q\|_{H^{-1}(\mathbb{R}^n)}^2 
\\
&\hskip 1.5cm \le s^{-2}+s^n\tau^{-2}+\epsilon^\beta s^n e^{\varkappa (s+\tau)} + s^ne^{c/\epsilon}e^{\varkappa (s+\tau)}  \| \Lambda_{q_1,\lambda}^0-\Lambda_{q_2,\lambda}^0\|^\theta.
\end{align*}
Taking in this inequality $s=\tau^{2/(n+2)}\le \tau$, we derive, upon substituting $2\varkappa$ by $\varkappa$,
\[
C\mathbf{m}_\lambda^{-2}\|q\|_{H^{-1}(\mathbb{R}^n)}^2
\le \tau^{-4/(n+2)}+\epsilon^\beta  e^{\varkappa \tau} + e^{c/\epsilon}e^{\varkappa \tau}  \| \Lambda_{q_1,\lambda}^0-\Lambda_{q_2,\lambda}^0\|^\theta.
\]
Let $\tau_0$ be sufficiently large in such a way that $\tau_0^{-4/(n+2)}e^{-\varkappa\tau_0}<1$. In this case for each $\tau > \tau_0$, there exists $0<\epsilon <1$ so that $\epsilon^\beta=\tau^{-4/(n+2)}e^{-\varkappa\tau}$. Whence, for $\tau>\tau_1=\max(2\varpi,\tau_0)$, taking $\epsilon=[\tau^{-4/(n+2)}e^{-\varkappa\tau}]^{1/\beta}$ in the preceding inequality, we find, upon modifying the constants above,
\[
C\mathbf{m}_\lambda^{-2}\|q\|_{H^{-1}(\mathbb{R}^n)}^2
\le \tau^{-4/(n+2)}+ e^{e^{\varkappa \tau}}  \| \Lambda_{q_1,\lambda}^0-\Lambda_{q_2,\lambda}^0\|^\theta.
\]
Hence, modifying again the different constants, we get an estimate of the form
\[
C\mathbf{m}_\lambda^{-1}\|q\|_{H^{-1}(\mathbb{R}^n)}
\le \tau^{-2/(n+2)}+ e^{e^{\varkappa \tau} }\| \Lambda_{q_1,\lambda}^0-\Lambda_{q_2,\lambda}^0\|^\theta.
\]
Let $\Upsilon_0$ be given by $\Upsilon_0^{-\theta}=\tau_1^{2/(n+2)}e^{e^{\varkappa \tau_1}}$ and assume that $0<\| \Lambda_{q_1,\lambda}^0-\Lambda_{q_2,\lambda}^0\|<\Upsilon_0$. Taking $\tau>\tau_1$ is the last inequality so that $\tau^{-2/(n+2)}= e^{\varkappa \tau}  \| \Lambda_{q_1,\lambda}^0-\Lambda_{q_2,\lambda}^0\|^\theta$, we obtain
\[
\|q_1-q_2\|_{H^{-1}(\mathrm{M})}\le C \mathbf{m}_\lambda \left|\ln \ln \| \Lambda_{q_1,\lambda}^0-\Lambda_{q_2,\lambda}^0\|\right|^{-2/(n+2)}.
\]
The proof of Theorem \ref{mt} is then complete.

\subsection{Proof of Theorem \ref{mt1}}

In this subsection $C>0$ still denotes a generic constant only depending on $n$, $\mathrm{M}$, $\mathrm{M}_0$, $\Gamma$, $\Sigma$, $\lambda_0$ and $\kappa$.

For $q\in \mathscr{Q}_\lambda$, set
\[
\mathscr{S}_{q,\lambda}=\{ u\in H^2(\mathrm{M});\; (\Delta +\lambda -q)u=0\}.
\]

Pick $q_1,q_2\in \mathscr{Q}_\lambda$ so that $q_1=q_2$ in $\mathrm{M}_1$. Let $u_j\in \mathscr{S}_{q_j,\lambda}$, $j=1,2$, be as above, where $\mathrm{M}_0$ is substituted by $\mathrm{M}$. Let $\tilde{u}_2\in H^2(M)$ be the solution of the BVP
\[
(\Delta +\lambda -q_1)\tilde{u}_2=0\; \mathrm{in}\; \mathrm{M},\quad \tilde{u}_2{_{|\partial M}}= u_2{_{|\partial M}},
\]
and $u=u_2-\tilde{u}_2=R_{q_1}(\lambda)((q_1-q_2)u_2)$.

Fix $\psi \in C_0^\infty (\mathrm{M}'_0)$ so that $\psi=1$ in a neighborhood of $\mathrm{M}_0$ and let $w=\psi u$. Then we have
\[
(\Delta +\lambda -q_1)w= \psi (q_2-q_1)u_2+[\Delta,\psi]u.
\]
Hence, Green's formula yields
\[
0=\int_M\psi(q_2-q_1)u_2u_1dx+\int_M[\Delta,\psi]uu_1dx.
\]
Taking into account that $\mathrm{supp}([\Delta,\psi]u)\subset U\Subset \mathrm{Int}(\mathrm{M}_1)$, we get
\[
\left|\int_M(q_2-q_1)u_2u_1dx\right|\le C_0\|u\|_{H^1(U)}\|u_1\|_{L^2(\mathrm{M})},
\]
where the constant $C_0>0$ only depend on $n$, $\mathrm{M}$, $\mathrm{M}_0$ and $\mathrm{M}'_0$.

Let $\rho>0$. Using Theorem \ref{theorem1.0}, in which $\mathrm{M}$ is substituted by $\mathrm{M}_1$, and Lemma \ref{lemma1}, we get
\[
\left|\int_M(q_2-q_1)u_2u_1dx\right|\le C\tilde{\mathbf{n}}_\lambda\left[\rho^p\|u_2\|_{L^2(\mathrm{M})}+ \rho^{-1}\| \partial_\nu u\|_{L^2(\Sigma)}\right]\|u_1\|_{L^2(\mathrm{M})},
\]
where $p$ is as in Theorem \ref{theorem1.0} and 
\[
\tilde{\mathbf{n}}_\lambda=\lambda \mathbf{b}_\lambda \mathbf{e}_\lambda=\lambda^{-1}\tilde{\mathbf{m}}_\lambda.
\]

Since
\begin{align*}
\| \partial_\nu u\|_{L^2(\Sigma)}&=\| (\Lambda_{q_1,\lambda}^1-\Lambda_{q_2,\lambda}^1)(u_2{_{|\partial M}})\|_{L^2(\Sigma)}
\\
&\le \| \Lambda_{q_1,\lambda}^1-\Lambda_{q_2,\lambda}^1\| \|u_2\|_{H^{3/2}(\partial \mathrm{M})},
\end{align*}
we derive
\[
\left|\int_M(q_2-q_1)u_2u_1dx\right|
 \le C\tilde{\mathbf{n}}_\lambda \left[\rho^p+\rho^{-1}  \| \Lambda_{q_1,\lambda}^1-\Lambda_{q_2,\lambda}^1\|\right]\|u_1\|_{L^2(\mathrm{M})}\|u_2\|_{H^2(\mathrm{M}_0)}.
\]

Taking in this inequality $\rho=\| \Lambda_{q_1}^1-\Lambda_{q_2}^1\|^{1/(p+1)}$, we get
\[
\left|\int_M(q_2-q_1)u_2u_1dx\right|\le C\tilde{\mathbf{n}}_\lambda \| \Lambda_{q_1,\lambda}^1-\Lambda_{q_2,\lambda}^1\|^\theta \|u_1\|_{L^2(\mathrm{M})}\|u_2\|_{H^2(\mathrm{M})},
\]
where $\theta=p/(p+1)$.

Proceeding similarly to the proof of Theorem \ref{mt}, we derive
\[
C\tilde{\mathbf{m}}_\lambda^{-1}\|q\|_{H^{-1}(\mathbb{R}^n)}
\le \tau^{-2/(n+2)}+ e^{\varkappa \tau}  \| \Lambda_{q_1,\lambda}^1-\Lambda_{q_2,\lambda}^1\|^\theta.
\]
Let $\Upsilon_1$ be given by $\Upsilon_1^{-\theta}=\tau_\ast^{2/(n+2)}e^{\varkappa \tau_\ast}$, where $\tau_\ast=2\varpi$. When $0<\| \Lambda_{q_1,\lambda}^1-\Lambda_{q_2,\lambda}^1\|<\Upsilon_1$, taking $\tau>\tau_\ast$ in the last inequality so that $\tau^{-2/(n+2)}= e^{\varkappa \tau}  \| \Lambda_{q_1,\lambda}-\Lambda_{q_2,\lambda}\|^\theta$, we find
\[
\|q_1-q_2\|_{H^{-1}(\mathrm{M})}\le C \tilde{\mathbf{m}}_\lambda\left|\ln \| \Lambda_{q_1,\lambda}^1-\Lambda_{q_2,\lambda}^1\|\right|^{-2/(n+2)}.
\]
This is the expected inequality.

\section{The interior impedance problem}\label{section7}

In this section we explain the main\linebreak modifications we bring to the proofs we carried out in the case of Dirichlet boundary condition to obtain a variant of Theorems \ref{mt} and \ref{mt1} in the case of an impedance boundary condition. A similar inverse problem to the one we consider in this section was already studied in \cite{KU}.

Since our results use \cite[Theorem 1.8]{BSW}, we assume in this section that $\mathrm{M}$ is a $C^\infty$ bounded domain of $\mathbb{R}^n$. Fix $a\in C^\infty (\partial \mathrm{M},\mathbb{R})$ such that $\pm a>0$ and consider the BVP
\begin{equation}\label{ip1}
(\Delta +\mu -q)u=f\; \mathrm{in}\; M,\quad (\partial_\nu u\mp i\sqrt{\mu}\, a)u=\varphi\; \mathrm{on}\; \partial M,
\end{equation}
where $\mu>0$, $q\in L^\infty(\mathrm{M},\mathbb{R})$, $f\in L^2(M)$ and $\varphi \in L^2(\partial \mathrm{M})$.

We proceed as in the proof of \cite[Proposition A1]{KU} to prove that the BVP \eqref{ip1} has a unique solution $u=u^\pm_{q,\mu}(f,\varphi )\in H^1(\mathrm{M})$. Furthermore, according to \cite[Theorem 1.8]{BSW}, we have
\begin{equation}\label{ip2}
\|\nabla u\|_{L^2(M)}+\sqrt{\mu} \|u\|_{L^2(\mathrm{M})}\le C_0\left(\|qu+f\|_{L^2(\mathrm{M})}+\|\varphi\|_{L^2(\partial \mathrm{M})} \right),
\end{equation}
where the constant $C_0$ only depends on $n$, $\mathrm{M}$ and $a$.

Let $\kappa >0$ be fixed and assume that $\|q\|_{L^\infty(\mathrm{M})}\le \kappa$. If $\mu \ge \mu_0:=4\kappa^2$ then \eqref{ip2} implies
\begin{equation}\label{ip3}
\|\nabla u\|_{L^2(M)}+\sqrt{\mu} \|u\|_{L^2(\mathrm{M})}\le C_1\left(\|f\|_{L^2(\mathrm{M})}+\|\varphi\|_{L^2(\partial \mathrm{M})} \right).
\end{equation}
Here and henceforth, constant $C_1>0$ is a generic constant only depending on $n$, $\mathrm{M}$, $\kappa$ and $a$.

Next, assume that $\varphi \in H^{1/2}(\partial \mathrm{M})$. We write $u=v+iw$ and we verify that $v$  is the solution of the BVP
\[
(\Delta +\mu -q)v=\Re f\; \mathrm{in}\; M,\quad \partial_\nu v=\mp a\sqrt{\lambda}\, w+\Re \varphi.
\]
It follows from \cite[Theorem 2.4.2.7]{Gr} that $v\in H^2(\mathrm{M})$ and \cite[Theorem 2.3.3.6]{Gr} yields
\begin{align}
\|v\|_{H^2(\mathrm{M})}&\le C_1\left(\mu\|u\|_{L^2(\mathrm{M})}+ \|f\|_{L^2(\mathrm{M})}+\|\mp a\sqrt{\lambda}\, w+\varphi\|_{H^{1/2}(\partial \mathrm{M})} \right)\label{ip4}
\\
&\le C_1\left(\mu\|u\|_{L^2(\mathrm{M})}+ \|f\|_{L^2(\mathrm{M})}+\sqrt{\lambda}\, \|w\|_{H^1(\mathrm{M})} +\| \varphi\|_{H^{1/2}(\partial \mathrm{M})} \right).\nonumber
\end{align}

We have similarly $w\in H^2(\mathrm{M})$ and the following inequality holds
\begin{equation}\label{ip5}
\|w\|_{H^2(\mathrm{M})}\le C_1\left(\mu\|u\|_{L^2(\mathrm{M})}+ \|f\|_{L^2(\mathrm{M})}+\sqrt{\lambda}\, \|v\|_{H^1(\mathrm{M})} +\| \varphi\|_{H^{1/2}(\partial \mathrm{M})} \right).
\end{equation}

Hence $u_{q,\mu}^\pm(f,\varphi)\in H^2(\mathrm{M})$ and putting together \eqref{ip3}, \eqref{ip4} and \eqref{ip5}, we obtain
\begin{equation}\label{ip6}
\|u^\pm_{q,\mu}(f,\varphi)\|_{H^2(\mathrm{M})}\le C_1\sqrt{\mu}\left( \|f\|_{L^2(\mathrm{M})}+\|\varphi\|_{H^{1/2}(\partial \mathrm{M})} \right).
\end{equation}

In the sequel $\mu \ge \mu_0$ is fixed and let
\[
\mathscr{Q}=\left\{ q\in L^\infty (\mathrm{M},\mathbb{R});\; \|q\|_{L^\infty(\mathrm{M})}\le \kappa\right\}.
\]

Define $H^1(\Sigma)$ as follows 
\[
H^1(\Sigma)=\{ f\in L^2(\Sigma);\; \nabla_\tau f\in L^2(\Sigma)\},
\]
where $\nabla_\tau f$ denotes the tangential gradient of $f$. We endow $H^1(\Sigma)$ with its natural norm
\[
\|f\|_{H^1(\Sigma)}=\left(\|f\|_{L^2(\Sigma)}^2+\|\nabla_\tau f\|_{L^2(\Sigma)}^2\right)^{1/2}.
\]
For further use, we observe that if $f\in H^1(\partial M)$, then $f_{|\Sigma}\in H^1(\Sigma)$ and 
\[
\|f_{|\Sigma}\|_{H^1(\Sigma)}\le \|f\|_{H^1(\partial \mathrm{M})},
\]
Here and henceforth, $H^1(\partial \mathrm{M})$ is endowed with the norm
\[
\|f\|_{H^1(\partial \mathrm{M})}=\left(\|f\|_{L^2(\partial \mathrm{M})}^2+\|\nabla_\tau f\|_{L^2(\partial \mathrm{M})}^2\right)^{1/2}.
\]

We associate to $q\in \mathscr{Q}$ the partial boundary operators $\mathscr{N}_q^0$ and $\mathscr{N}_q^1$ given respectively as follows
\begin{align*}
&\mathscr{N}_{q,\mu}^0:H_\Gamma^{1/2}(\partial \mathrm{M})\rightarrow H^1(\Sigma): \varphi \mapsto u^+_{q,\mu}(0,\varphi)_{|\Sigma},
\\
&\mathscr{N}_{q,\mu}^1:H^{1/2}(\partial \mathrm{M})\rightarrow H^1(\Sigma): \varphi \mapsto u^+_{q,\mu}(0,\varphi)_{|\Sigma}.
\end{align*}

It follows from \eqref{ip6} that
\[
\mathscr{N}_{q,\mu}^0\in \mathscr{B}(H_\Gamma^{1/2}(\partial \mathrm{M}),H^1(\Sigma))\quad \mathrm{and}\quad  \mathscr{N}_{q,\lambda}^1\in \mathscr{B}(H^{1/2}(\partial \mathrm{M}),H^1(\Sigma)).
\]

Define
\[
\mathscr{H}=\{u\in H^2(\mathrm{M});\; (\partial_\nu \mp i\sqrt{\mu}\, a)u_{|\Sigma}=0\}.
\]

Minor modifications of the proof of Theorem \ref{theorem1.0} enable us to obtain the following result
\begin{theorem}\label{theoremip1}
Let $U\Subset \mathrm{Int}(\mathrm{M})$. For any $q\in \mathscr{Q}$, $\epsilon >0$ and $u\in \mathscr{H}$, we have
\[
C\|u\|_{H^1(U)}\le \sqrt{\mu}\mathbf{b}_\mu\left[\epsilon ^p\|u\|_{H^1(\mathrm{M})}
+\epsilon^{-1} \left( \|(\Delta -q+\mu )u\|_{L^2(\mathrm{M})}+\|u\|_{H^1(\Sigma )}\right)\right],
\]
where the constant $C>0$ and $p>0$ only depend on $n$, $\mathrm{M}$, $\kappa$, $\Sigma$ and $U$.
\end{theorem}

In light of this quantitative unique continuation result, we obtain by mimicking the proof of Theorem \ref{mt1} the following stability inequality.

\begin{theorem}\label{theoremip3.0}
There exists $0<\Upsilon <1$ only depending on $n$, $\mathrm{M}$, $\kappa$, $\Sigma$  so that for every $q_1,q_2\in \mathscr{Q}$ satisfying $q_1=q_2$ in $\mathrm{M}_1$ and $0<\|\mathscr{N}_{q_1,\mu}^1-\mathscr{N}_{q_2,\mu}^1\|<\Upsilon$,we have
\[
\|q_1-q_2\|_{H^{-1}(\mathrm{M})}\le C\sqrt{\mu}\mathbf{b}_\mu \left| \ln \|\mathscr{N}_{q_1,\mu}^1-\mathscr{N}_{q_2,\mu}^1\|\right|^{-2/(n+2)},
\]
where the constant $C>0$ only depends on $n$, $\mathrm{M}$, $\kappa$, $\Sigma$.
\end{theorem}

In the rest of this section $n=3$ and we assume in addition that $\lambda_0=\mu_0$. For all $q\in \mathscr{Q}_\lambda$, recall that
\[
\mathscr{S}_{q,\lambda}^0=\{ u\in H^2(\mathrm{M}_0);\; (\Delta +\lambda -q)u=0\},
\]
and set
\[
\mathcal{S}_{q,\lambda}^1=\{ u\in H^2(\mathrm{M});\; (\Delta +\lambda -q)u=0,\; (\partial_\nu -i\sqrt{\lambda}\, a)u\in H^{1/2}_\Gamma(\partial \mathrm{M})\}.
\]

Let $\Gamma_0\Subset \mathrm{Int}(\Gamma)$,  $\mathcal{E}=H^{1/2}(\Gamma)$ and $\psi \in C^\infty (\partial \mathrm{M})$ satisfying $0\le \psi \le 1$, $\psi=1$ in $\overline{\Gamma_0}$ and $\psi=0$ in $\partial M\setminus \Gamma$. Then consider the closed subset of $\mathcal{E}$ given by
\[
\mathcal{F}=\{f\in \mathcal{E};\; \psi f=0\}.
\]
The quotient space $\mathcal{E}/\mathcal{F}$ will be endowed with its quotient norm 
\[
\|\varphi\|_{\mathcal{E}/\mathcal{F}}=\inf\{\|f\|_{\mathcal{E}};\; f\in \dot{\varphi}\},
\]
where
\[
\dot{\varphi}=\{f\in \mathcal{E};\; \pi(f)=\varphi\}
\]
and $\pi:\mathcal{E}\rightarrow \mathcal{E}/\mathcal{F}$ is the quotient map. Also, if necessary, we identify $(\mathcal{E}/\mathcal{F})^\ast$ with $\mathcal{F}^\bot=\{f^\ast \in \mathcal{E}^\ast;\; f^\ast_{|\mathcal{F}}=0\}$, via the isometric isomorphism $f^\ast\in (\mathcal{E}/\mathcal{F})^\ast\mapsto f^\ast \circ \pi \in \mathcal{F}^\bot$. 

Let $\mathcal{H}$ be the closure of $\mathscr{S}_{q,\lambda}^0$ in $L^2(\mathrm{M}_0)$. We verify that if $\varphi \in \mathcal{E}/\mathcal{F}$, then $u^+_{q,\lambda}(0,\psi\overline{f})$ is independent of $f\in \dot{\varphi}$. Therefore, we can  define the bounded operator
\[
\mathcal{T}:\mathcal{E}/\mathcal{F}\rightarrow \mathcal{H}:\varphi \mapsto u^+_{q,\lambda}(0,\psi\overline{f})_{|\mathrm{M}_0}.
\]

Let $\mathfrak{I}$ be the canonical isomorphism 
\[
\mathfrak{I}: (\mathcal{E}/\mathcal{F})^\ast\rightarrow \mathcal{E}/\mathcal{F}: \langle u,v\rangle =(\mathfrak{I}u|v)_{\mathcal{E}/\mathcal{F}},
\]
where $(\cdot |\cdot)_{\mathcal{E}/\mathcal{F}}$ is the scalar product of $\mathcal{E}/\mathcal{F}$.

Let $\varphi \in \mathcal{E}/\mathcal{F}$, $v\in L^2(\mathrm{M}_0)$ and $w=u^-_{q,\lambda}(v\chi,0)$. As we have done for $T$, making integrations by parts, we find the formula
\[
(\mathcal{T}\varphi,v)_{L^2(\mathrm{M}_0)}=(\varphi|\overline{\mathfrak{I}(\psi \overline{w}}))_{\mathcal{E}/\mathcal{F}}.
\]
This identity shows that $\mathcal{T}^\ast$, the adjoint of $\mathcal{T}$, is given as follows
\[
\mathcal{T}^\ast: \mathcal{H}\rightarrow \mathcal{E}/\mathcal{F}: v\mapsto \overline{\mathfrak{I}(\psi\overline{w})}.
\]

We can mimic the proof we carried out for $T$ to show that $\mathcal{T}^\ast\mathcal{T}:\mathcal{E}/\mathcal{F}\rightarrow \mathcal{E}/\mathcal{F}$ is compact self-adjoint and positive definite operator. In consequence, there exist a sequence of positive numbers $(\mu_j)$ and an orthonormal basis $(\psi_j)$ of $\mathcal{E}/\mathcal{F}$ so that
\[
\mathcal{T}^\ast\mathcal{T}=\mu_j\psi_j,\quad j\ge 1.
\]
Furthermore, if $u_j=\mu_j^{-1/2}\mathcal{T}\psi_j$, $j\ge 1$, then $(u_j)$ forms an orthonormal basis of $\mathcal{H}$.

We use in the sequel the following notation
\[
\mathbf{n}_\lambda=\lambda^3\mathbf{e}_\lambda.
\]

\begin{theorem}\label{theoremip2}
(Quantitative Range approximation) Let $q\in \mathscr{Q}_\lambda$. For every $0<\epsilon <1$ and $u\in \mathscr{S}_{q,\lambda}^0$ there exists $v\in \mathcal{S}_{q,\lambda}^1$ such that
\begin{equation}\label{rap1}
\| u-v_{|\mathrm{M}_0}\|_{L^2(\mathrm{M}_0)}\le C\epsilon^\beta  \mathbf{n}_\lambda\|u\|_{H^2(\mathrm{M}_0)},
\quad \|v\|_{H^2(\mathrm{M})}\le e^{c/\epsilon}\|u\|_{L^2(\mathrm{M}_0)},
\end{equation}
where the constants $C>0$, $c>0$ and $\beta>0$ only depend on $\mathrm{M}$, $\mathrm{M_0}$, $\Gamma$, $\kappa$ and $a$.
\end{theorem}

\begin{proof}
Pick $t>0$ and set $N_t^1=\{j;\; \tau_j>t\}$. Let $u=\sum_j a_ju_j\in \mathcal{H}$ and
\[
\varphi_t=\sum_{N_t} \tau_j^{-1} a_j\psi_j,
\]
where $\tau_j=\mu_j^{1/2}$ for each $j$.

Noting that
\[
\|\varphi_t\|_{\mathcal{E}/\mathcal{F}}^2=\sum_{N_t^1} \tau_j^{-2} |a_j|^2\le t^{-2}\sum_j |a_j|^2=t^{-2}\|u\|_{L^2(\mathrm{M}_0)}^2,
\]
we derive the following inequality
\begin{equation}\label{4e}
\|\varphi_t\|_{\mathcal{E}/\mathcal{F}}\le t^{-1}\|u\|_{L^2(\mathrm{M}_0)}.
\end{equation}

Let $N_t^0=\{j;\; \tau_j\le t\}$ and
\[
v_t =\sum_{j\in N_t^0} a_ju_j.
\]
Then $u=T\varphi_t+v_t$ and, since $T\varphi_t\bot v_t$, we find
\[
\|v_t\|_{L^2(\mathrm{M}_0)}^2=(v_t|v_t)_{L^2(\mathrm{M}_0)}=(v_t+Tg_t|v_t)_{L^2(\mathrm{M}_0)}=(u|v_t)_{L^2(\mathrm{M}_0)}.
\]
Let $w_t=u^-_{q,\lambda}(\chi v_t,0)$. Then
\[
\|v_t\|_{L^2(\mathrm{M}_0)}^2=(u|(\Delta +\lambda -q)w_t)_{L^2(\mathrm{M}_0)}.
\]
Using $(\Delta +\lambda -q)u=0$ and applying Green's formula, we obtain
\begin{align*}
\|v_t\|_{L^2(\mathrm{M}_0)}^2&=(u|\partial_\nu w_t)_{L^2(\partial \mathrm{M}_0)}-(\partial_\nu u|w_t)_{L^2(\partial \mathrm{M}_0)}
\\
&=-(u|iaw_t)_{L^2(\partial \mathrm{M}_0)}-(\partial_\nu u|w_t)_{L^2(\partial \mathrm{M}_0)}
\\
&\le \|a\|_{L^\infty(\partial \mathrm{M})}\|u\|_{L^2(\partial \mathrm{M}_0)} \|w_t\|_{L^2(\partial \mathrm{M}_0)}+\|\partial_\nu u\|_{L^2(\partial \mathrm{M}_0)}\|w_t\|_{L^2(\partial \mathrm{M}_0)}.
\end{align*}

As the trace mapping $h\in H^1(\mathrm{M}_1)\mapsto h_{|\partial \mathrm{M}_0}\in  L^2(\partial \mathrm{M}_0)$ is bounded, we get from the last inequality
\begin{equation}\label{6e}
\|v_t\|_{L^2(\mathrm{M}_0)}^2\le C_0\|u\|_{H^2(\mathrm{M}_0)}\|w_t\|_{H^1(\mathrm{M}_1)},
\end{equation}
where the constant $C_0>0$ only depends on $n$, $\mathrm{M}_1$, $n$ and $a$.

Pick $0<\epsilon <1$ arbitrary. Modifying slightly the proof Theorem \ref{theoremB1} in which $\mathrm{M}$ is substituted by $\mathrm{M}_1$ and using \eqref{ip6}, we obtain
\begin{align*}
\|w_t\|_{H^1(\mathrm{M}_1)}&\le C\lambda^{-1/2}\mathbf{n}_\lambda\left[ \epsilon ^\beta \|v_t\|_{L^2(\mathrm{M}_0)}+ e^{c/\epsilon}\| w_t\|_{H^1(\Gamma_0)}\right]
\\
&\le C\lambda^{-1/2}\mathbf{n}_\lambda\left[ \epsilon ^\beta \|v_t\|_{L^2(\mathrm{M}_0)}+ e^{c/\epsilon}\| \psi w_t\|_{H^1(\partial \mathrm{M})}\right]
\end{align*}
where the constants $C>0$, $c>0$ and $\beta$ only depend on $\mathrm{M}$, $\mathrm{M_0}$, $\Gamma$, $\kappa$ and $a$.
The interpolation inequality in \cite[Theorem 7.7]{LM} yields
\begin{align*}
\| \psi w_t\|_{H^1(\partial \mathrm{M})}&\le C \| \psi w_t\|_{H^{-1/2}(\partial \mathrm{M})}^{1/4}\| \psi w_t\|_{H^{1/2}(\partial \mathrm{M})}^{3/4}
\\
&\le C \| \psi w_t\|_{H^{-1/2}(\partial \mathrm{M})}^{1/4}\| \psi w_t\|_{H^{1/2}(\partial \mathrm{M})}^{3/4}
\\
&\le C \| \mathcal{T}^\ast v_t\|_{\mathcal{E}/\mathcal{F}}^{1/4}\| \psi w_t\|_{H^{1/2}(\partial \mathrm{M})}^{3/4}
\\
&\le Ct^{1/4} \|v_t\|_{L^2(\mathrm{M}_0)}^{1/4}\| w_t\|_{H^{1/2}(\partial \mathrm{M})}^{3/4}
\end{align*}
In light of \eqref{ip6}, the last inequality gives
\[
\| \psi w_t\|_{H^1(\partial \mathrm{M})}\le Ct^{1/4}\lambda^{1/2} \| v_t\|_{L^2(\mathrm{M}_0)}
\]

In consequence, we have
\[
\|w_t\|_{H^1(\mathrm{M}_1)}\le C\mathbf{n}_\lambda\left[\epsilon ^\beta+ t^{1/4}e^{c/\epsilon}\right]\|v_t\|_{L^2(\mathrm{M}_0)},\quad  0<\epsilon <1.
\]
This and \eqref{6e} yield
\[
\|v_t\|_{L^2(\mathrm{M}_0)}\le C\mathbf{n}_\lambda\|u\|_{H^2(\mathrm{M}_0)}\left[\epsilon ^\beta + t^{1/4}e^{c/\epsilon}\right],\quad  0<\epsilon <1,
\]

Taking  $t^{1/4}=\epsilon^\beta e^{-c/\epsilon}$ in the last inequality, we  get
\begin{equation}\label{9e}
\|v_t\|_{L^2(\mathrm{M}_0)}\le C \epsilon^\beta \mathbf{n}_\lambda\|u\|_{H^2(\mathrm{M}_0)},
\end{equation}
and in light of \eqref{4e} we derive, upon substituting $c$ by another similar constant,
\begin{equation}\label{10e}
\|\varphi_t\|_{\mathcal{E}/\mathcal{F}}\le \epsilon^{c/\epsilon}\|u\|_{L^2(\mathrm{M}_0)}.
\end{equation}

Let $v=u^+_{q,\lambda}(0,\psi f_t)$, where $f\in \dot{\varphi}$ is arbitrary chosen. Then $v\in \mathscr{S}_{q,\lambda}^1$ and  as an immediate consequence of \eqref{ip6}  and \eqref{10e}, we get
\[
\|v\|_{H^2(\mathrm{M})}\le \epsilon^{c/\epsilon}\|u\|_{L^2(\mathrm{M}_0)}.
\]
The proof is then complete.
\end{proof}

Next, we prove the following proposition, where
\[
\tilde{\mathbf{n}}_\lambda=\max\left(\mathbf{n}_\lambda^2, \sqrt{\lambda}\mathbf{b}_\lambda\right).
\]

\begin{proposition}\label{propositionip1}
(Integral inequality) Let $q_1,q_2\in \mathscr{Q}_\lambda$ satisfying $q_1=q_2$ in $\mathrm{M}_1$. For each $0<\epsilon <1$ and $u_j\in \mathscr{S}_{q_j,\lambda}^0$, $j=1,2$, we have
\begin{align*}
&C\tilde{\mathbf{n}}_\lambda^{-1}\left|\int_{\mathrm{M}_0}(q_1-q_2)u_1u_2d\mu\right|\le \epsilon^\beta  \|u_1\|_{H^2(\mathrm{M}_0)}\|u_2\|_{H^2(\mathrm{M}_0)}
\\
&\hskip 4.5cm+e^{c/\epsilon}\|\mathscr{N}_{q_1,\lambda}^0-\mathscr{N}_{q_2,\lambda}^0\|^\theta \|u_1\|_{L^2(\mathrm{M})}\|u_2\|_{L^2(\mathrm{M})},
\end{align*}
where the constants $C>0$ and $c>0$ only depend on $\mathrm{M}$, $\mathrm{M}_0$, $\kappa$,  $\Gamma$ and $\Sigma$, while the constant $\beta>0$ only depends on $\mathrm{M}$, $\mathrm{M}_0$, $\kappa$,  and $\Gamma$, and the constant $0<\theta<1$ only depends on $\mathrm{M}$, $\mathrm{M}_0$, $\kappa$,  and $\Sigma$.
\end{proposition}

\begin{proof}
 In this proof $C>0$, $c>0$ and $\beta>0$ are generic constants having the same dependence as in the statement of Proposition \ref{propositionip1}.

Let $0<\epsilon <1$ and $u_j\in \mathscr{S}_{q_j,\lambda}^0$, $j=1,2$. From Theorem \ref{theoremip2}, there exist $v_j\in \mathcal{S}_{q_j,\lambda}^1$, $j=1,2$, so that
\begin{align}
&\| u_j-v_j{_{|\mathrm{M}_0}}\|_{L^2(\mathrm{M}_0)}\le C\epsilon^\beta  \mathbf{n}_\lambda\|u_j\|_{H^2(\mathrm{M}_0)},\label{rap2}
\\
& \|v_j\|_{H^2(\mathrm{M})}\le e^{c/\epsilon}\|u_j\|_{L^2(\mathrm{M}_0)}.\label{rap3}
\end{align}
Similarly to \eqref{2.4} we get from \eqref{rap2}
\begin{align}
&C\left|\int_{\mathrm{M}_0}(q_1-q_2)u_1u_2d\mu\right|\le \epsilon^\beta \mathbf{n}_\lambda^2 \|u_1\|_{H^2(\mathrm{M}_0)}\|u_2\|_{H^2(\mathrm{M}_0)} \label{2.4.0}
\\
&\hskip 6cm+\left|\int_{\mathrm{M}_0}(q_1-q_2)v_1v_2d\mu\right|.\nonumber
\end{align}
Let $\tilde{v}_2$ be the solution of the BVP
\[
(\Delta +\lambda -q_1)\tilde{v}_2=0\; \mathrm{in}\; \mathrm{M}\quad (\partial_\nu -i\sqrt{\lambda}\, a)\tilde{v}_2=(\partial_\nu -i\sqrt{\lambda}\, a)v_2.
\]
Pick $\psi \in C_0^\infty (\mathrm{M}_0')$ so that $\psi=1$ in a neighborhood of $\mathrm{M}_0$. Let $v=v_2-\tilde{v}_2$ and $w=\psi v$. Using that $w$ is the solution of the equation
\[
(\Delta +\lambda -q_1)w=(q_2-q_1)v_2+[\Delta,\psi]v\; \mathrm{in}\; \mathrm{M}.
\]
and Green's formula, we obtain
\[
\int_\mathrm{M}(q_2-q_1)v_1v_2dx=-\int_\mathrm{M}[\Delta,\psi]vv_1dx.
\]
As $\mathrm{supp}([\Delta,\psi]v)\subset U\Subset \mathrm{Int}(\mathrm{M}_1)$, we find
\begin{equation}\label{ip15}
\left| \int_\mathrm{M}(q_2-q_1)v_1v_2dx\right|\le C_0\|v\|_{H^1(U)}\|v_1\|_{L^2(\mathrm{M})}.
\end{equation}
On the other hand, applying Theorem \ref{theoremip1}, where $\mathrm{M}$ is substituted by $\mathrm{M}_1$, and using \eqref{ip3}, we obtain for every $\rho>0$
\begin{align*}
C\|v\|_{H^1(U)}&\le \rho ^p\sqrt{\lambda}\mathbf{b}_\lambda\left[ \|v_2\|_{L^2(\mathrm{M}_0)}+\rho^{-1}\|v\|_{H^1(\Sigma)}\right]
\\
&\le \rho ^p\sqrt{\lambda}\mathbf{b}_\lambda\left[ \|v_2\|_{L^2(\mathrm{M}_0)}+\rho^{-1}\|(\mathscr{N}_{q_1,\lambda}^0-\mathscr{N}_{q_2,\lambda}^0)((\partial_\nu -i\sqrt{\lambda}\, a)v_2)\|_{H^1(\Sigma)}\right]
\\
&\le \rho ^p\sqrt{\lambda}\mathbf{b}_\lambda \left[\|v_2\|_{L^2(\mathrm{M}_0)}+\rho^{-1}\|\mathscr{N}_{q_1,\lambda}^0-\mathscr{N}_{q_2,\lambda}^0\|\|(\partial_\nu -i\sqrt{\lambda}\, a)v_2)\|_{H^{1/2}(\partial \mathrm{M})}\right],
\end{align*}
where $p$ is as in Theorem \ref{theoremip1}. This inequality in \eqref{ip15} implies
\begin{align*}
&\left| \int_\mathrm{M}(q_2-q_1)v_1v_2dx\right|
\\
&\hskip 1cm \le C\sqrt{\lambda}\mathbf{b}_\lambda\left[ \rho ^p \|v_2\|_{L^2(\mathrm{M})}+\rho^{-1}\|\mathscr{N}_{q_1,\lambda}^0-\mathscr{N}_{q_2,\lambda}^0\|\|v_2\|_{H^2(\mathrm{M})}\right] \|v_1\|_{L^2(\mathrm{M})}.
\end{align*}
Combined with \eqref{rap3}, this inequality yields
\begin{align*}
&\left| \int_\mathrm{M}(q_2-q_1)v_1v_2dx\right|
\\
&\hskip 1cm \le \sqrt{\lambda}\mathbf{b}_\lambda e^{c/\epsilon}\left[ \rho ^p +\rho^{-1}\|\mathscr{N}_{q_1,\lambda}^0-\mathscr{N}_{q_2,\lambda}^0\|\right] \|u_1\|_{L^2(\mathrm{M})}\|u_2\|_{L^2(\mathrm{M})}.
\end{align*}
Taking $\rho=\|\mathscr{N}_{q_1,\lambda}^0-\mathscr{N}_{q_2,\lambda}^0\|^{1/(1+p)}$, we get
\[
\left| \int_\mathrm{M}(q_2-q_1)v_1v_2dx\right|\le \sqrt{\lambda}\mathbf{b}_\lambda e^{c/\epsilon}\|\mathscr{N}_{q_1,\lambda}^0-\mathscr{N}_{q_2,\lambda}^0\|^\theta \|u_1\|_{L^2(\mathrm{M})}\|u_2\|_{L^2(\mathrm{M})},
\]
where $\theta=p/(1+p)$.

The expected inequality is obtained by putting together the last inequality and \eqref{2.4.0}.
\end{proof}

In light of the integral inequality in Proposition \ref{propositionip1}, we can repeat step by step the proof of Theorem \ref{mt}. We obtain the following result.

\begin{theorem}\label{theoremip3}
Assume $n=3$. There exists $0<\Upsilon <e^{-1}$ only depending on $\mathrm{M}$, $\mathrm{M}_0$, $\Gamma$, $\Sigma$ and $\kappa$   so that  for every $q_1,q_2\in \mathscr{Q}_\lambda$ satisfying $q_1=q_2$ in $\mathrm{M}_1$ and $0<\|\mathscr{N}_{q_1,\lambda}^0-\mathscr{N}_{q_2,\lambda}^0\|<\Upsilon$, we have
\[
\|q_1-q_2\|_{H^{-1}(\mathrm{M})}\le C\tilde{\mathbf{n}}_\lambda \left| \ln \ln \|\mathscr{N}_{q_1,\lambda}^0-\mathscr{N}_{q_2,\lambda}^0\|\right|^{-2/(n+2)},
\]
where the constant $C>0$ only depends on $\mathrm{M}$, $\mathrm{M}_0$, $\Gamma$, $\Sigma$ and $\kappa$.
\end{theorem}


\end{document}